\RequirePackage{fix-cm}
\documentclass[smallextended]{svjour3}       % onecolumn (second format)

\def\makeheadbox{{}}

\smartqed  % flush right qed marks, e.g. at end of proof
\usepackage{graphicx}
\usepackage{lipsum}
\usepackage{amsmath,amssymb}
\usepackage{graphicx}
\usepackage{hhline,url}
\usepackage[12pt]{extsizes}
\usepackage{geometry}
\usepackage{algorithm,algpseudocode}
\algnewcommand{\Initialize}[1]{%
	\State \textbf{Initialize:}
	\State \hspace*{\algorithmicindent}\parbox[t]{0.8\linewidth}{\raggedright #1}
}

\usepackage{mathptmx}      % use Times fonts if available on your TeX system
\newtheorem{thm}{Theorem}
\newtheorem{lem}[thm]{Lemma}
\newenvironment{rem}{\begin{remark}\rm}{\end{remark}}
\newtheorem{prop}[thm]{Proposition}
\newtheorem{Definition}[thm]{Definition}
\newenvironment{dfn}{\begin{Definition}\rm}{\end{Definition}}
\newtheorem{Corollary}[thm]{Corollary}

\newtheorem{Example}{Example}[section]

\newtheorem{algor}[thm]{Method}

\newtheorem{Condition}[thm]{Condition}

\newcommand{\R}{\mathbb{R}}
\newcommand{\dd}{\,\mathrm{d}}
\newcommand{\ve}{\varepsilon}

\renewcommand{\phi}{\varphi}

\newcommand{\bm}[1]{{\mbox{\boldmath $#1$}}}

\DeclareMathOperator{\cv}{conv}

\DeclareMathOperator{\ri}{ri}
\DeclareMathOperator{\supp}{supp}
\DeclareMathOperator{\aff}{aff}

\newcommand{\E}[1]{\mathrm{E}\left[#1\right]}

\newcommand{\ds}{\mathrm{d}}
\newcommand{\cd}{\circ\ds}
\newcommand{\ccd}{\circ\,\ds}
\newcommand{\A}{\mathcal{A}}

\usepackage{xcolor}

\journalname{Japan Journal of Industrial and Applied Mathematics}
\begin{document}

\title{Monte Carlo construction of cubature on Wiener space%\thanks{Grants or other notes
}


\author{Satoshi Hayakawa\and
	Ken'ichiro Tanaka}


\institute{Satoshi Hayakawa \at
	Mathematical Institute,
	University of Oxford,
	Andrew Wiles Building,
	Radcliffe Observatory Quarter,
	Woodstock Road,
	Oxford,
	OX2 6GG, United Kingdom\\
	\email{hayakawa@maths.ox.ac.uk}           %  \\
	\and
	Ken'ichiro Tanaka \at
	Graduate School of Information Science and Technology,
	The University of Tokyo,
	7-3-1, Hongo, Bunkyo-ku, Tokyo, Japan
}
\date{}

\maketitle

\begin{abstract}
	In this paper, we investigate application of mathematical optimization
	to construction of a cubature formula on Wiener space,
	which is a weak approximation method of stochastic differential equations
	introduced by Lyons and Victoir
	(Cubature on Wiener Space, Proc. R. Soc. Lond. A 460, 169--198).
	After giving a brief review on the cubature theory on Wiener space,
	we show that a cubature formula of general dimension and degree
	can be obtained through a Monte Carlo sampling and linear programming.
	This paper also includes an extension of stochastic Tchakaloff's theorem,
	which technically yields the proof of our primary result.
\end{abstract}

\section{Introduction}
%%\nopar

Cubature on Wiener space \cite{lyo04}
is a certain family of numerical formula
for approximating the expectation of functionals
of diffusion processes,
which are important in mathematical finance and other related fields.
The existence of cubature formula of general dimension and degree
has been known,
but the constructions given in the literature were
based on algebraic structure among continuous paths
and Brownian motion,
and limited to ones of low degree and dimension.
Our aim in this paper is to obtain a method of constructing
general cubature formula on Wiener space
through mathematical optimization.

More concretely,
we are interested in approximating values in the form $\E{f(X_T)}$,
where $(X_t)_{0\le t\le T}$
is a solution of stochastic differential equation (SDE)
driven by a multidimensional Brownian motion,
and $f$ is a function with some regularity (e.g., Lipschitz continuity).
The standard approach consists of two steps:
\begin{itemize}
	\item[(1)]
		approximate the distribution of $X_{t+\Delta t}$
		from the information of (approximated) $X_t$,
		where $\Delta t \ll 1$;
	\item[(2)]
		split the time interval $[0, T]$ by
		$0 = t_0 < t_1 < \cdots < t_k = T$ and
		sequentially apply (1) over each $[t_\ell, t_{\ell+1}]$.
\end{itemize}
For example, the stochastic process
$\tilde{X}^{\mathrm{EM}, (k)}_T$ given
by the Euler--Maruyama method with equal partition attains
$\left\lvert \mathrm{E}\,[f(\tilde{X}^{\mathrm{EM}, (k)}_T)]
- \E{f(X_T)} \right\rvert = \mathrm{O}(1/k)$
with respect to $k$, the number of partitions
\cite{klo92}.

To achieve higher precision of approximation such as
$\mathrm{O}(1/k^2)$,
Lyons and Victoir \cite{lyo04} introduced
the cubature on Wiener space,
which uses the higher order of stochastic Taylor expansion
in step (1).
The basic scheme for (1) in cubature on Wiener space is as follows:
\begin{description}
	\item[(c1)]
		approximate the distribution of Brownian motion
		$(B_s)_{t\le s\le t+\Delta t}$
		by a weighted discrete set of deterministic paths;
	\item[(c2)]
		locally (over $[t, t+\Delta t]$)
		solve ordinary differential equations (ODEs)
		driven by the deterministic paths in (c1)
		instead of the stochastic differential equations
		for approximating the distribution of $X_{t+\Delta t}$;
\end{description}
Among the steps (c1), (c2), (2) in cubature on Wiener space,
this paper primarily focuses on finding a good weighted set
of deterministic paths described in (c1).

In the field of high-order approximation of SDEs,
there is also a relative of cubature on Wiener space
called Kusuoka approximation \cite{kus04} using a random paths
in (c1),
which is followed by the papers \cite{nin08,nin09}
presenting concrete second-order schemes.
The main challenge shared by these approaches
(cubature on Wiener space, Kusuoka approximation)
is that constructing such a (random) set of paths
has been performed by actually solving Lie-algebraic equations
and is limited to low dimension and degree of precision.
Therefore, our objective is to find a way of generally
constructing such a formula in arbitrary dimension and degree.

\paragraph{Contribution of this study}
Broadly speaking,
the contribution of this study comprises the following two items:
\begin{itemize}
	\item (main contribution)
	We show that one can construct a cubature formula on Wiener space
	of general dimension and degree with a randomized algorithm.
	\item (technical contribution)
	To apply the technique of \cite{hayakawa-MCCC} to the problem
	of cubature on Wiener space,
	we characterize the affine hull
	of the distribution of iterated Stratonovich integrals
	and prove stochastic Tchakaloff's  theorem
	in a stronger way.
\end{itemize}

The main result with a simple Monte Carlo construction is given
in Proposition \ref{prop:pw-generate}.
It asserts that a certain random generation of
piecewise linear paths combined with a linear programming yields
a cubature formula on Wiener space.

As a technical contribution, we extend stochastic Tchakaloff's  theorem
\cite{lyo04}, which assures the existence of cubature formulas on Wiener space.
Although the original statement was just that there exists a cubature formula,
we show that the expectation of iterated Stratonovich integrals
of a Brownian motion ($\E{\bm\phi_\mathrm{W}(B)}$ in \eqref{eq:CoW})
is contained in the relative interior of
$\cv\{\bm\phi_\mathrm{W}(w)\}$ with a valid range of bounded variation (BV) paths $w$
(Theorem \ref{st-tch}).
This stronger statement with ``relative interior"
follows our characterization (Proposition \ref{BM-full}) of
$\aff\supp\mathrm{P}_{\bm\phi_\mathrm{W}(B)}$ in terms of \eqref{eq:CoW}
and is essential in exploiting the existing construction
of general cubature formula \cite{hayakawa-MCCC}.

We only treat the part (c1) in this study,
but it is important to consider combining our construction
with ODE solvers (c2) and some techniques reducing computational complexity
(2)
such as recombination \cite{lit12},
which needs further investigation and is deferred to future research.

\paragraph{Outline}
We give a brief overview of the following sections.

Section \ref{chap:prep} describes the idea and background of
this study in a more mathematical way.
We briefly explain the concept of cubature on Wiener space
in Section \ref{subsec:CoW} and the overview of a preceding result
recently given by one of the authors \cite{hayakawa-MCCC}
in Section \ref{prelim}.

Section \ref{chap:CoW} is devoted to a theoretical review of
cubature on Wiener space by \cite{lyo04}.
We introduce the facts around vector fields and
relevant notations in Section \ref{sec:vector-fields} used throughout the paper.
The precise definition and error estimate of cubature on Wiener space
is given in Section \ref{sec:formulation},
and Section \ref{sec:const} provides information about
known constructions based on algebraic arguments.

In Section \ref{sec:st-tch}, we give the extended statement
of stochastic Tchakaloff's theorem and its proof.
Sections \ref{sec:algebra} and \ref{sec:signature} provide
algebraic background of cubature construction.
Section \ref{proof-stt} is devoted to the proof of our version of stochastic
Tchakaloff's theorem, and it also includes the characterization
of the distribution of iterated Stratonovich integrals
from the viewpoint of the affine hull of the support.

We discuss a way to obtain piecewise linear cubature formula on Wiener space
in Section \ref{chap:optCoW}.
After giving some general properties of continuous BV paths
in our context in Section \ref{sec:sig-BV},
we prove our main result in Section \ref{MCWC} that
we can generally construct cubature formulas on Wiener space
with simple randomized algorithms.
Section \ref{num-exp} represents the numerical verification of our result
in a small range of parameters.

Finally we summarize our conclusion in Section \ref{sum-dis}.

%%%%%%%%%%%%%%%%%%%%%%%%%%%%%%%%%%%%%%%%%
%%%%%%%%%%%%%%%%%%%%%%%%%%%%%%%%%%%%%%%%%

\section{Preliminaries}\label{chap:prep}

In this section, we shall briefly explain the background and motivation
of this study.
Section \ref{subsec:CoW} gives a brief explanation of the cubature theory
on Wiener space \cite{lyo04} and identifies the problem to be solved.
Section \ref{prelim} introduces the recent Monte Carlo approach by
\cite{hayakawa-MCCC} to general cubature construction,
which our method in this paper is based on. We explain the relation
between cubature on Wiener space and a generalized one, as well as
the particular difficulty that arise on the Wiener space.

\subsection{Cubature on Wiener space}\label{subsec:CoW}
Let $B=(B^1, \ldots, B^d)$ be a $d$-dimensional standard Brownian motion.
Let $C_b^\infty(\R^N; \R^N)$ be the space of infinitely differentiable $\R^N$-valued functions defined on $\R^N$ whose every order of derivative is bounded.
Let us consider the following $N$-dimensional Stratonovich SDE:
\begin{equation}
	\ds X_t = \sum_{i=1}^d V_i(X_t) \cd B_t^i + V_0(X_t)\dd t,
	\qquad X_0=x, \label{sde}
\end{equation}
where $x\in\R^N$, $V_i\in C_b^\infty(\R^N;\R^N)$ for $i=0,\ldots,d$.
As the process $X_t$ is dependent on the initial value $x$,
we denote it by $X_t(x)$ if necessary.
We may assume the solution $X_t(x)$
is continuous with respect to $t$ and $x$.
Our aim is to efficiently compute or approximate the expectation
$\E{f(X_t)}$ with $t>0$ and some smooth or Lipschitz $f$.
This sort of approximation is called a {\it weak approximation} of SDE
and well-studied in the literature \cite{klo92,kus01,kus04,lyo04,nin09,nin08,nin03}.

We here focus on the approach introduced in
\cite{lyo04} called {\it cubature on Wiener space}.
Broadly speaking, a cubature formula on Wiener space (of the time interval
$[0, T]$)
is the approximation
\[
\mathrm{P}\simeq \sum_{i=1}^n \lambda_j \delta_{w_j},
\]
where $\mathrm{P}$ is the Wiener measure on the Wiener space
$C_0^0([0, T]; \R^d)$ (the space of $\R^d$-valued continuous
function in $[0, T]$ starting at the origin), 
\[
w_j=(w_j^0, w_j^1,\ldots, w_j^d)\in C_0^0([0, T]; \R\oplus\R^d)
\]
for $j=1,\ldots, n$, and
$\lambda_1,\ldots,\lambda_n$ are positive real weights whose sum equals one.
Instead of polynomials in conventional cubature formulas,
we adopt iterated integrals as the test functionals, that is,
we want to find paths $w_i$ satisfying
\begin{equation}
	\E{\int_{0<t_1<\cdots<t_k<T}\circ\,\ds B_{t_1}^{i_1}
		\cdots \cd B_{t_k}^{i_k}}
	=\sum_{j=1}^n
	\lambda_j\int_{0<t_1<\cdots<t_k<T}\ds w^{i_1}_j(t_1)
	\cdots \ds w_j^{i_k}(t_k)
	\label{intro-cubature}
\end{equation}
over some set of multiindices $(i_1,\ldots,i_k)$,
where the iterated Stratonovich integral appears in the left-hand side.
Precisely speaking, we formally set $B^0_t=t$ for $t\ge0$ and
assume $w_j$ is a path of BV for each $j=1,\ldots,n$.
Although $w_j^0(t)=t$ is also assumed in \cite{lyo04},
here we may generalize and remove this condition.

The iterated integrals appearing in \eqref{intro-cubature} have
a rich algebraic structure (see Section \ref{sec:algebra}),
so algebraic approaches have been adopted in the literature
\cite{lyo04,nin08,gyu11,osh12,shi17,nin19}.
However, solving complicated equations of Lie algebra is required
in those approaches, and constructions of the formula are limited
to a small range (see Section \ref{sec:const}).
Our objective is to give a construction method for a general setting
where there are no limitations on the number of iterations of the integral
($k$ in \eqref{intro-cubature}).
For this purpose, we adopt an optimization-based viewpoint instead of algebraic ones
and extend to our situation the result of \cite{hayakawa-MCCC}, which gives
a randomized construction of general cubature formula.

We should mention the Kusuoka approximation \cite{kus01,kus04,shi17},
which is closely related to cubature on Wiener space.
However, our objective in this paper is limited to the construction
of the cubature on Wiener space, and the optimization-based approach
to the general Kusuoka approximation is deferred for future work.

\subsection{Monte Carlo approach to generalized cubature}\label{prelim}

A cubature formula is originally a numerical integration formula
on some Euclidean space
that exactly integrates polynomials up to a certain degree
\cite{ste16}, the theory of which underlies cubature on Wiener space
introduced in the previous section.
Then, we shall explain it in a generalized setting
and briefly explain the idea of \cite{hayakawa-MCCC} for constructing
general cubature formulas.

Let $(\Omega, \mathcal{G})$
be some measurable space and $X$
be a random variable on it.
A generalized cubature formula with respect to $X$
and integrable functions
$\phi_1, \ldots, \phi_D:\Omega\to\R$ is a set of points
$x_1,\ldots,x_n\in \Omega$
and positive weights, $\lambda_1,\ldots,\lambda_n$ such as 
\[
\E{\phi_i(X)}=\sum_{j=1}^n\lambda_j\phi_i(x_j), \quad i=1,\ldots,D.
\]
For simplicity, we then assume $\phi_1\equiv1$.
In this setting, $\lambda_1+\cdots+\lambda_n=1$ must hold.
We can also regard the above condition as one vector-valued equality
$\E{\bm\phi(X)}=\sum_{j=1}^n\lambda_j\bm\phi(x_j)$,
where $\bm\phi:\Omega\to\R^D$ is defined as 
$\bm\phi=(\phi_1,\ldots,\phi_D)^\top$.
The existence of such formula is assured by the following theorem
\cite{tch57,bay06,EDT}:
\begin{thm}[Generalized Tchakaloff's theorem]\label{MCCC:thm:gen-tchakaloff}
	Under the above setting, there exists a cubature formula
	whose number of points satisfies $n\le D$.
	Moreover, we can take points $x_1,\ldots,x_n$ so as to satisfy
	$\bm\phi(x_j)\in\supp\mathrm{P}_{\bm\phi(X)}$ for each $j=1,\ldots,n$.
\end{thm}

In the above statement, $\supp\mathrm{P}_{\bm\phi(X)}$ is the support
of the distribution of the vector-valued random variable $\phi(X)$.
Equivalently, this is the smallest closed set $A\subset\R^D$ satisfying
$\mathrm{P}(\bm\phi(X)\in A) = 1$.
Generalized Tchakaloff's theorem can be understood as an immediate
consequence of a discrete-geometric argument.
Indeed, $\E{\bm\phi(X)}$ is contained in the convex hull of
$\supp\mathrm{P}_{\bm\phi(X)}$
(the convex hull of an $A\subset\R^D$ is defined as 
$\cv A :=\{\sum_{i=1}^m\lambda_i x_i \mid m\ge1,\ \lambda_i\ge0,
\ \sum_{i=1}^m\lambda_i=1,\ x_i\in A\}$).
Therefore, the generalized Tchakaloff's theorem follows 
Carath\'{e}odory's theorem
(note that we assume $\phi_1\equiv1$):
\begin{thm}[Carath\'{e}odory]
	For an arbitrary $A\subset\R^D$ and $x\in \cv A$,
	there exists $D+1$ points $x_1, \ldots, x_{D+1}\in A$
	such that $x\in \cv\{x_1,\ldots,x_{D+1}\}$.
\end{thm}

Although the above argument cannot directly
be used in the construction of cubature,
we can use its nature by introducing the concept of relative interior.
For a set $A\subset\R^D$, its affine hull is defined by
\[
\aff A:=\left\{
\sum_{i=1}^m \lambda_ix_i \mid m\ge1,\ \lambda_i\in\R,\
\sum_{i=1}^m \lambda_i=1,\ x_i\in A
\right\}
\]
Then, the relative interior of $A$ is the interior of $A$
regarding the subspace topology on $\aff A$
and denoted by $\ri A$.
In terms of this relative interior, the following generalization of
Carath\'{e}odory's theorem holds:
\begin{thm}[\cite{ste16,bon63}]\label{MCCC:thm:ri}
	Suppose an $A\subset\R^D$
	satisfies $\aff A$ is a $k$-dimensional affine subspace of $\R^D$.
	Then, for each $x\in\ri\cv A$,
	there exists some subset $B\subset A$ composed of at most $2k$ points
	such that
	$\aff B=\aff A$ and $x\in\ri\cv B$.
\end{thm}

From this generalization and the fact that
\begin{equation}
	\E{\bm\phi(X)}\in\ri\cv\supp\mathrm{P}_{\bm\phi(X)}
	\label{MCCC:lem:exp-ri}
\end{equation}
holds
(see, e.g., \cite{hayakawa-MCCC} or essentially \cite{bay06} for proof),
we obtain the following randomized construction
of cubature formulas from i.i.d. copies of $X$:
%\nopar
\begin{thm}[\cite{hayakawa-MCCC}]\label{MCCC:thm:main}
	Let $X_1, X_2, \ldots$ be i.i.d. copies of $X$.
	Then there exists almost surely a positive integer $n$ satisfying
	$\E{\bm\phi(X)} \in \cv \{\bm\phi(X_1), \ldots, \bm\phi(X_n)\}$.
\end{thm}
%%\repar

Though the weights remain undetermined, for a sufficiently large $n$,
it suffices to take a basic feasible solution of the linear programming problem
\[
\text{minimize}\ \  0 \qquad \text{subject to}\ \ 
\sum_{j=1}^n\lambda_j \bm\phi(X_j) = \E{\bm\phi(X)},\ \lambda_j\ge0.
\]
Indeed, its basic feasible solution satisfies the bound of points used
in a cubature given in Tchakaloff's theorem
(Theorem \ref{MCCC:thm:gen-tchakaloff}).
This sort of technique reducing the number of points in a discrete measure
is called Carath\'{e}odory-Tchakaloff subsampling \cite{pia17}.

Here, if we formally write the iterated integral
$\int_{0<t_1<\cdots<t_k<T}\ds w^{i_1}(t_1)
\cdots \ds w^{i_k}(t_k)$  appearing
in \eqref{intro-cubature}
as $\phi_{(i_1, \ldots,i_k)}(w)$ for a valid $w$
(and the Brownian motion $B$),
then the cubature on Wiener space is a set of paths $w_j$
and weights $\lambda_j$ formally satisfying
\begin{equation}
	\E{\bm\phi_\mathrm{W}(B)}
	= \sum_{j=1}^n \lambda_j \bm\phi_\mathrm{W}(w_j),
	\label{eq:CoW}
\end{equation}
where $\bm\phi_\mathrm{W}$ denotes a vector of some functions
of the form $\phi_{(i_1,\ldots,i_k)}$.
Therefore, if we could directly generate sample paths of
the Brownian motion, then Theorem \ref{MCCC:thm:main} should be applicable.
In reality, it is impossible to generate a Brownian motion on a computer,
and it is not even a BV path.
However, we assume the following variant,
supporting our arguments.

\begin{rem}\label{rem-assumption}
	The assumption that $X_1, X_2, \ldots$
	possess the same distribution as $X$
	in Theorem \ref{MCCC:thm:main} can be relaxed;
	the same conclusion yields from the following condition
	for the i.i.d. sequence:
	\begin{equation}
		\aff\supp\mathrm{P}_{\bm\phi(X_1)}=\aff\supp\mathrm{P}_{\bm\phi(X)},
		\qquad
		\supp\mathrm{P}_{\bm\phi(X_1)} \supset \supp\mathrm{P}_{\bm\phi(X)}.
		\label{reduced-condition}
	\end{equation}
	From this fact, it is sufficient to investigate the distribution
	of iterated integrals,
	and we indeed show the Wiener space counterpart of
	the condition \eqref{reduced-condition} in
	Proposition \ref{BM-full}.
\end{rem}

%%%%%%%%%%%%%%%%%%%%%%%%%%%%%%%%%%%%%%%%%%%%%%%%%%%%%%%%%%%%%%%%%%%%%%%%%%
%%%%%%%%%%%%%%%%%%%%%%%%%%%%%%%%%%%%%%%%%%%%%%%%%%%%%%%%%%%%%%%%%%%%%%%%%%

\section{Theoretical background of cubature on Wiener space}\label{chap:CoW}

In this section, we provide a theoretical
review on the cubature theory on Wiener space,
first introduced by \cite{lyo04}.
We quickly introduce basic notions concerning multidimensional stochastic flows,
and give the error estimate of cubature formula.
Moreover,
we will demonstrate few examples of concrete construction of cubature formula
on Wiener space in Section \ref{sec:const}.

%%%%%%%%%%%%%%%%%%%%%%%%%%%%%%%%%%%%%%%%%%%%%%%%%%%%%5
\subsection{Vector fields}\label{sec:vector-fields}
In this section, we define vector fields on $\R^N$ and show the correspondence
between vector fields and vector-valued functions.
Let $C^\infty(\R^N)$ be the set of real-valued smooth functions over $\R^N$.

\begin{dfn}
	A {\it vector field} on $\R^N$ is a ($\R$-)linear mapping
	$V:C^\infty(\R^N)\to C^\infty(\R^N)$ such that
	$V(fg)=(Vf)g + fVg$ holds for arbitrary $f, g\in C^\infty(\R^N)$.
\end{dfn}

Due to this condition, a vector field on $\R^N$ has to be
a differential operator
$\sum_{i=1}^NV^i\partial_i$ where $V^i\in C^\infty(\R^N)$ and
$\partial_i$ denotes the $i$-th partial derivative
for $i=1,\ldots,d$.
Therefore, a vector field corresponds to
the vector-valued smooth function
$(V^1, \ldots, V^N)^\top : \R^N \to \R^N$.
By abuse of notation, we also denote this vector-valued function
by $V$.

If $A$ and $B$ are vector fields on $\R^N$,
we define the {\it Lie bracket} $[A, B]:=AB-BA$.
This $[A, B]$ is also a vector field
because the second derivatives vanish.
Note that $[A, B]$ corresponds to the vector
$(\partial B)A -(\partial A)B$, where $A$, $B$ are regarded as functions
and $\partial C$ denotes the Jacobian matrix of $C$ (see, e.g., \cite{hai11}).

Reciprocally, the coefficients of the SDE (\ref{sde})
can be regarded as vector fields.
These vector fields are closely related to the behavior of $X_t$.

Let $V_0, \ldots, V_d$ be the vector fields (operators) induced by
the coefficients of \eqref{sde}, and define the operator
$L:=V_0+\frac12(V_1^2+\cdots+V_d^2)$.
Let us consider the parabolic partial differential equation (PDE)
\begin{equation}
\begin{cases}
\displaystyle\frac{\partial}{\partial t}u(t, x) = Lu(t, x),\\
\hspace{3.75mm} u(0, x) = f(x),
\end{cases}
\label{pde}
\end{equation}
with a Lipschitz function $f:\R^N\to\R^N$.
Because  $u(T, x)=\E{f(X_T(x))}$ holds \cite{ike89},
we can exploit the numerical schemes in PDE theory to get $\E{f(X_T(x))}$
and vice versa.

%%%%%%%%%%%%%%%%%%%%%%%%%%%%%%%%%%%%%%%%%%%%%%%%%

We also introduce several conditions
on the vector fields.
They are assumed to obtain
the estimate given in the Proposition \ref{prop:estimate}.
Before we state those, we introduce some notations
based on \cite{kus01}.

Let $\A:=\{\emptyset\}\cup\bigcup_{k=1}^\infty\{0, 1,\ldots,d\}^k$.
For $\alpha\in\A$, define $|\alpha|:=0$ if $\alpha=\emptyset$ and
$|\alpha|:=k$ if $\alpha=(\alpha_1,\ldots,\alpha_k)\in\{0,\ldots,d\}^k$.
We also define $\|\alpha\|:=|\alpha|+|\{1\le j\le |\alpha|
\mid \alpha_j=0\}|$.
For $\alpha,\beta\in\A$,
define $\alpha*\beta:=
(\alpha_1,\ldots,\alpha_{|\alpha|},\beta_1,\ldots,\beta_{|\beta|})$.
Let $\A_0:=\A\setminus\{\emptyset\}$ and
$\A_1:=\A\setminus\{\emptyset\cup (0)\}$.
We also define, for each integer $m\ge1$,
\[
\A(m):=\{\alpha\in\A\mid \|\alpha\|\le m\},
\quad \A_0(m):=\A(m)\cap\A_0,
\quad \A_1(m):=\A(m)\cap\A_1.
\]

Define a vector field $V_{[\alpha]}$ for each $\alpha\in\A$
inductively by
$V_{[\emptyset]}=0$ and
\begin{align*}
	&V_{[i]}(=V_{[(i)]}) := V_i\quad (i=0,\ldots,d),\\
	&V_{[\alpha*(i)]} := [V_\alpha, V_i] \quad
	(|\alpha|\ge1, i=0,\ldots,d).
\end{align*}

We can now state the uniformly finitely generated (UFG) condition
\cite{kus87,kus01}:
\begin{quote}
	(UFG) There exists a positive integer $L\ge1$ such that,
	for an arbitrary $\alpha\in\A_1$, there exists
	$\phi_{\alpha,\beta}\in C_b^\infty(\R^N)$ for each $\beta\in \A_1(L)$
	satisfying
	\[
	V_{[\alpha]}=\sum_{\beta\in\A_1(L)}\phi_{\alpha,\beta}V_{[\beta]}.
	\]
\end{quote}
This is equivalent to the statement that the $C_b^\infty(\R^N)$-module
generated by $\{V_{[\alpha]}\mid \alpha\in\A_1\}$ is finitely generated.
Note that (UFG) is known to be strictly weaker than the uniform
H\"{o}rmander condition (see, e.g., Example 2 in \cite{kus03}),
which is one of the typical assumptions on vector fields.

Although only the condition (UFG) was assumed in \cite{lyo04},
it was pointed out in \cite{cri07} that the following condition
is also essential:
\begin{quote}
	(V0) There exists $\phi_\beta\in C_b^\infty(\R^N)$
	for each $\beta\in\A_1(2)$ such that
	\[
	V_0=\sum_{\beta\in\A_1(2)} \phi_\beta V_{[\beta]}.
	\]
\end{quote}

For a function $f\in C_b^\infty(\R^N)$, define $(P_tf)(x):=\E{f(X_t(x))}$.
The following estimate is essential.
%\nopar
\begin{prop}[\cite{kus87,cri07}]\label{prop:estimate}
	Assume that both (UFG) and (V0) hold.
	Then, for any positive integer $r$ and $\alpha_1,\ldots, \alpha_r\in \A$,
	there exists a constant $C>0$ such that
	\begin{equation}
	\|V_{[\alpha_1]}\cdots V_{[\alpha_r]}P_tf\|_\infty
	\le \frac{Ct^{1/2}}{t^{(\|\alpha_1\|+\cdots+\|\alpha_r\|)/2}}
	\|\nabla f\|_\infty
	\label{eq:estimate}
	\end{equation}
\end{prop}
%%\repar

Although we can obtain a weaker bound without assuming (V0),
we later exploit this bound assuming both (UFG) and (V0) for simplicity.

%\subsection{Stochastic Taylor formula}

We finally state the stochastic Taylor formula in terms
of the vector-field notation introduced above.
By It\^{o}'s formula, we obtain for any $f\in C_b^\infty(\R^N)$
\begin{align*}
	f(X_t)&= f(x)+ \sum_{i=0}^d \sum_{j=1}^N \int_0^t V_i^j(X_s)
	\partial_jf(X_s) \cd B^i_s\\
	&=f(x)+\sum_{i=0}^d\int_0^t (V_if)(X_s)\cd B^i_s,
\end{align*}
where we denote $\ds s$ by $\circ\,\ds B^0_s$.
Therefore, the repetition of Ito's formula yields
\[
f(X_t)=f(x)+\sum_{i=0}^d (V_if)(x) \int_{0<s<t} \ccd B_s^i
+ \sum_{i,j=0}^d (V_iV_jf)(x) \int_{0<t_1<t_2<1}
\ccd B^i_{t_1}\cd B^j_{t_2}
\cdots.
\]
This is the stochastic Taylor formula,
which is rigorously stated as follows.
For a multiindex
$\alpha=(\alpha^1,\ldots,\alpha^k)\in\A$,
we denote by $V_\alpha$ the operator $V_{\alpha^1}\cdots V_{\alpha^k}$.
%\nopar
\begin{prop}[{\cite[Proposition 2.1]{lyo04}}; 
	\cite{klo92}]\label{prop:st-taylor}
	Let $f\in C_b^\infty(\R^N)$ and $m$ be a positive integer.
	Then, we have
	\[
	f(X_t(x))=\sum_{\alpha\in\A(m)}
	(V_\alpha f)(x)
	+R_m(t, x, f),
	\]
	where the remainder term satisfies, for some constant $C>0$,
	\[
	\sup_{x\in \R^N}\sqrt{\E{R_m(t, x, f)^2}}
	\le C t^{\frac{m+1}2}\sup_{\beta\in\A(m+2)\setminus\A(m)}
	\|V_\beta f\|_\infty.
	\]
\end{prop}
%%\repar

\subsection{Formulation and evaluation of cubature on Wiener space}
\label{sec:formulation}
We can now precisely define the cubature formula
\cite{lyo04}.
\begin{dfn}
	Let $T>0$, and let $m$ be a positive integer.
	BV paths $w_1,\ldots,w_n\in C_0^0([0, T]; \R\oplus\R^d)$
	and weights $\lambda_1,\ldots,\lambda_n\ge0$
	with $\sum_{i=1}^n\lambda_j=1$ to define
	a cubature formula on Wiener space of degree $m$ at time $T$,
	if only
	\[
	\E{\int_{0<t_1<\cdots<t_k<T}
		\ccd B^{i_1}_{t_1}\cdots \cd B^{i_k}_{t_k}}
	=\sum_{j=1}^n \lambda_j \int_{0<t_1<\cdots<t_k<T}
	\ds w_{j}^{i_1}(t_1)\cdots \ds w_{j}^{i_k}(t_k).
	\]
	holds for all $(i_1,\ldots,i_k)\in\A(m)$.
	Here, $\ccd B^0_s$ represents $\ds s$.
\end{dfn}

To construct a cubature formula, it suffices to find it over $[0, 1]$.
Indeed, when $w_1,\ldots,w_n\in C_0^0([0, 1]; \R\oplus\R^d)$
form the cubature over $[0, 1]$,
\[
w_{T, j}^i(t):=\begin{cases}
Tw_j^i(t/T) & (i=0)\\
\sqrt{T}w_j^i(t/T) & (i=1,\ldots,d)
\end{cases}
\]
with the same weights define
the cubature over $[0, T]$.
This is an immediate consequence of the scaling property of
the Brownian motion.

Once such paths are given, we can easily compute each evolution driven by
$w_i$ as it is just an ODE.
For a BV path $w\in C_0^0([0, T]; \R\oplus\R^d)$,
define $\tilde{X}_{t}(x, w)$ as the solution of the ODE
\[
\ds \tilde{X}_t(x, w)=\sum_{i=0}^d V_i(\tilde{X}_t(x, w))\dd w^i (t),
\quad
\tilde{X}_0(x, w)=x.
\]
Then, $\sum_{j=1}^n\lambda_jf(\tilde{X}_T(x, w_{T,j}))$ should approximate well
$\E{f(X_T(x))}$.
Indeed, the estimate in Proposition \ref{prop:st-taylor} holds for $t=T$
if we replace the Wiener measure by the discrete measure
$\sum_{j=1}^n \lambda_j \delta_{w_j}$.
Therefore, by applying Cauchy-Schwarz we obtain the evaluation:
\begin{equation}
\sup_{x\in\R^N}\left\lvert
\E{f(X_T(x))}-
\sum_{j=1}^n\lambda_j f(\tilde{X}_T(x, w_{T,j}))\right\rvert
\le CT^{\frac{m+1}2}\sup_{\beta\in\A(m+2)\setminus\A(m)}
\|V_\beta f\|_\infty
\label{est1}
\end{equation}
with the constant $C>0$ depending only on $w_1,\ldots,w_n$.

The above formula does not work
as a good approximation unless $T$ is small.
Therefore, we divide $[0, T]$ into smaller time intervals as
$0=t_0<t_1<\cdots<t_k=T$.
If we consider the repeated application of the cubature formula
over each subinterval $[t_{\ell-1}, t_\ell]$, we can use
\begin{equation}
\sum_{j_1,\ldots,j_k=1}^n
\lambda_{j_1}\cdots\lambda_{j_k}
f(X_T(x, w_{s_1, j_1}*\cdots*w_{s_k, j_k})),
\label{eq:rep-cubature}
\end{equation}
where $w*v$ denotes the concatenation of two paths
and $s_\ell:=t_\ell-t_{\ell-1}$ for each $\ell=1,\ldots k$,
as an approximation of the expectation $\E{f(X_T(x))}$.
If we define discrete
Markov random variables $Y_0,\ldots,Y_k$
independent of the Brownian motion
as
\[
Y_0=x, \qquad \mathrm{P}(Y_\ell=\tilde{X}_{s_\ell, y}
(w_{s_\ell, j}) \mid Y_{\ell-1}=y)
=\lambda_j
\quad (\ell=1,\ldots,k,\ j=1,\ldots, n),
\]
$\E{Y_k}$ coincides with the approximation \eqref{eq:rep-cubature}.
Then, combining an estimate for
\[
\sup_{x\in\R^N}\left\lvert
\E{f(Y_k)\mid Y_0=x}-\E{f(X_T(x))}
\right\rvert
\]
with Proposition \ref{prop:estimate},
we can prove the following assertion.

%\nopar
\begin{prop}[{\cite[Proposition 3.6]{lyo04}}]\label{prop:comb-cub}
	Let $f$ be a bounded Lipschitz function in $\R^N$.
	Then, under (UFG) and (V0), we have
	\[
	\sup_{x\in\R^N}\left\lvert
	\E{f(Y_k)\mid Y_0=x}-\E{f(X_T(x))}
	\right\rvert
	\le C\|\nabla f\|_\infty\left(s_k^{1/2}+\sum_{\ell=1}^{k-1}
	\frac{s_\ell^{(m+1)/2}}{(T-t_\ell)^{m/2}}\right)
	\]
	for some constant $C>0$,
	which is dependent only on $m$ and $w_1,\ldots,w_n$.
\end{prop}
%%\repar

The equally spaced partition $t_\ell=\ell T/k$ for $\ell=0,\ldots, k$
is not the optimal one in terms of
asymptotic error bound with $k\to\infty$.
Consider taking $t_\ell = T\left(1-\left(1
-\frac{\ell}k\right)^\gamma\right)$
with a constant $\gamma>0$ independent of $k$
($\gamma=1$ corresponds to the equally spaced partition).
By taking $\gamma > m-1$, we have the following estimate
\cite{kus01}:
\[
\sup_{x\in\R^N}\left\lvert
\E{f(Y_k)\mid Y_0=x}-\E{f(X_T(x))}
\right\rvert
\le C k^{-(m-1)/2}\|\nabla f\|_\infty.
\]
Therefore, a cubature formula of degree $m$ with an appropriate time partition
achieves the error rate $\mathrm{O}(k^{-(m-1)/2})$ where $k$ is the number of
partitions.

\begin{rem}
	If we have to compute all the $k$-times concatenation
	of a cubature formula composed of $n$ sample paths,
	we have to solve $\frac{n^{k+1}-1}{n-1}$ ODEs in total \cite{lyo04}.
	When number of ODEs is too large, we reduce the computational complexity through some Monte Carlo
	simulation or subsampling method
	\cite{lit12,tch15,pia17}.
	In this paper, we do not consider efficient implementation
	of concatenation of cubature formula,
	but only consider their constructions.
\end{rem}

\subsection{Known constructions of cubature on Wiener space}\label{sec:const}

We should note that some concrete examples of cubature formulas
on Wiener space are already known.
The simplest case treated in \cite{lyo04} is $m=3$,
where we have a cubature formula composed of
linear paths (i.e., with only one linear segment).

Let $n:=2^d$ and $z_1, \ldots, z_n \in \R^d$ be all the elements of $\{-1, 1\}^d$.
Then, paths 
\[
w_i(t) := t(1, z_i) = (t, tz_i^1, \ldots, tz_i^d),\qquad
0\le t\le 1, \quad i=1,\ldots,n
\]
with weights $\lambda_1=\cdots=\lambda_n=2^{-d}$
construct a cubature formula with $m=3$.
Although $2^d$ is much larger than $|\mathcal{A}(3)| = \mathrm{O}(d^3)$,
we can reduce the number of paths, e.g., using
Carath\'{e}odory-Tchakaloff subsampling.

Constructions for the case $d$ in general and $m=5$ are also given
in \cite{lyo04}, where the authors give
cubature formula using only $\mathrm{O}(d^3)$ paths.
Moreover, \cite{gyu11} constructed higher order cubature formula up to $m=11$,
but the construction is limited to one dimensional space-time ($d=1$).
Ref. \cite{nin19} represents other concrete examples when $m=5$ with general $d$
and the case $(d, m)=(2, 7)$.

All the aforementioned examples are derived
by solving equations in terms of Lie algebra (see Section \ref{sec:algebra}),
which can be directly
written using the Campbell-Baker-Hausdorff formula.
However, as a different approach, we address
an optimization-based construction
in the next section.

%%%%%%%%%%%%%%%%%%%%%%%%%%%%%%%%%%%%%%%%%%%%%%%%%%%%%%%%%%%%%%
%%%%%%%%%%%%%%%%%%%%%%%%%%%%%%%%%%%%%%%%%%%%%%%%%%%%%%%%%%%%%%
%%%%%%%%%%%%%%%%%%%%%%%%%%%%%%%%%%%%%%%%%%%%%%%%%%%%%%%%%%%%%%

\section{Stochastic Tchakaloff's theorem}\label{sec:st-tch}

In the previous section, we have demonstrated
the theoretical support of cubature on Wiener space.
However, it is important to know if such formulas
can actually be constructed.
In this section, we shall state stochastic Tchakaloff's theorem,
which
assures the existence of cubature formula on Wiener space.
Though the stochastic Tchakaloff's theorem is originally given in \cite{lyo04},
we state it in a stronger way by using the concept of relative interior.

Before doing so, we shall introduce the rich algebraic structures behind
the theory of cubature on Wiener space in the following two sections,
which are also essential in our proof of stochastic Tchakaloff's theorem.

%%%%%%%%%%%%%%%%%%%%%%%%%%%%%%%%%%%%%%%%%%%%%%%%%%%%%%%%%%%%%

\subsection{Tensor and Lie algebra}\label{sec:algebra}
We introduce a tensor algebra which is suitable to our case
\cite{lyo98,lyo04,kus04}.
Denote $\R\oplus\R^d$ by $E$.
Define $U_0(E):=\R\ (=: E^{\otimes0})$.
Let $A_0:=\R$ and $A_1:=\R^d$,
and define
\[
U_n(E):=\bigoplus_{
	\substack{(i_1,\ldots,i_k)\in \{0, 1\}^k,\\
		2k-(i_1+\cdots+i_k)=n}
}
A_{i_1}\otimes \cdots \otimes A_{i_k}
\]
for each positive integer $n$.
Here, the condition for $(i_1,\ldots,i_k)$ means that
$A_{i_1}\otimes\cdots\otimes A_{i_k}$ takes all the arrangement
of $\R$ and $\R^d$ such that
$2(\#\text{ of }\R)+(\#\text{ of } \R^d)=n$.
Then, we consider the tensor algebra of formal series
\[
T((E)):=\bigoplus_{n=0}^\infty U_n
\left(
\simeq \bigoplus_{n=0}^\infty E^{\otimes n}
\right),
\]
{\it where the direct sum hereafter is regarded as a series},
i.e., $T$ is the set of all the {\it infinite} sequences
$(a_n)_{n=0}^\infty$ where $a_n\in U_n$ for each $n\ge0$.
Let $T^{(n)}(E):=\bigoplus_{k=0}^nU_k$,
and let $\pi_n:T((E))\to T^{(n)}(E)$ be the canonical projection
for each $n\ge0$.
As we are only interested in these projections in practice,
we do not need to differentiate the usual tensor algebra from
that of series treated here \cite{rou99}.

It might be easier to understand $T$ as the ring of formal power series
$\R[[Z_0, Z_1,\ldots,Z_d]]$ with noncommutative variables $Z_0,\ldots,Z_d$
(see, e.g., \cite{bau04}).
In that case, we redefine the degree of some monomial
$Y$ by $\deg Y:=2\deg_{Z_0}Y + (\deg_{Z_1}Y+\cdots+\deg_{Z_d}Y)$
(where each $\deg_{Z_i}Y$ denotes the number of $Z_i$ appearing in $Y$)
and regard $U_n(E)$ as the subspace spanned by monomials of degree $n$
for each $n\ge0$.

For any elements $a=(a_n)_{n=0}^\infty, b=(b_n)_{n=0}^\infty \in T((E))$,
we define the sum and product as follows:
\[
a+b:=(a_n+b_n)_{n=0}^\infty,\qquad
a\otimes b:=\left(\sum_{i=0}^n
a_i\otimes b_{n-i}\right)_{n=0}^\infty.
\]
The action by scalar is element-wise.
These definitions are straightforward
if we consider $\R[[Z_0,Z_1,\ldots,Z_d]]$.
Moreover, we define the exponential, inverse, and logarithm:
\[
\exp(a):=\sum_{k=0}^\infty \frac{a^{\otimes k}}{k!},
\quad
a^{-1}:=\frac1{a_0}\sum_{k=0}^\infty \left(1-\frac{a}{a_0}
\right)^{\otimes k},
\quad
\log a:=\log a_0 - \sum_{k=1}^\infty\frac1k\left(1-\frac{a}{a_0}
\right)^{\otimes k},
\]
where the latter two operations are limited for $a\in T((E))$
with $a_0\ne0$.
Note that these operations commute with each projection 
homomorphism $\pi_n$.

Let us introduce the space of Lie series.
Define
\[
L((E)):=
0\oplus E \oplus [E, E] \oplus [E, [E, E]] \oplus\cdots
\subset T((E)) \simeq \bigoplus_{n=0}^\infty E^{\otimes n},
\]
where, for linear subspaces $A, B\in T((E))$,
$[A, B]$ is the linear subspace of $T((E))$ spanned by
Lie brackets
$[a, b]:=a\otimes b-b\otimes a$ ($a\in A$, $b\in B$).
$L((E))$ is the so-called free Lie algebra generated by $E$
\cite{reu93}.
The elements of $L((E))$ are called Lie series.
We also define $L^{(n)}(E):=\pi_n(L((E)))$,
the elements of which are called Lie polynomials.

%Finally, let $G_n(E):=\exp L^{(n)}(E)=\{
%	\pi_n(\exp (a)) \mid a\in L^{(n)}(E)
%\}\subset T^{(n)}(E)$ for each $n\ge0$.
%As $a_0=0$ for each $a=(a_n)_{n=0}^\infty\in L((E))$,
%the $0$-th term of $\exp(a)$ equals $1$.
%The elements of $G_n(E)$ is actually invertible
%as $\exp(a)\otimes \exp(-a)=1$.
%Moreover, it is known that $G^{(n)}(E)$ is a group
%\cite{reu93}.
%A constructive formula deriving $\log (\exp(a)\otimes \exp(b))$
%using $a, b$ and Lie brackets is known and named as 
%the Baker-Campbell-Hausdorff formula,
%though we here do not give its explicit expression.

%%%%%%%%%%%%%%%%%%%%%%%%%%%%%%%%%%%%%%%%%%%%%%%%%%%%%%%%%%%%%

\subsection{Signature of a path}\label{sec:signature}

We shall introduce the signature
(or Chen series \cite{che57}) of a path,
which summarizes the algebraic structure of iterated integrals.
Let $w=(w^0,\ldots,w^d)\in C_0^0([0, T]; \R\oplus\R^d)$
be a BV path and 
we define its signature.
The integration by $\ds w$ is a 

\begin{dfn}
	For $0\le s\le t\le T$, define $S(w)_{s,t}\in T((E)) \simeq \R[[Z_0, Z_1,\ldots,Z_d]]$ by
	\begin{align*}
	S(w)_{s, t}
	&:=\sum_{n=0}^\infty \int_{s<t_1<\cdots<t_n<t}\ds w(t_1)
	\otimes\cdots\otimes\ds w(t_n)\\
	&:=\sum_{n=0}^\infty
	\sum_{(i_1,\ldots,i_k)\in\A(n)\setminus\A(n-1)}
	\left(\int_{s<t_1<\cdots<t_n<t}\ds w^{i_1}(t_1)
	\cdots\ds w^{i_k}(t_n)\right)
	Z_{i_1}\cdots Z_{i_k},
	\end{align*}
	where the integration by $\ds w$ means the
	Lebesgue--Stieltjes integration.
	In both presentations,
	we think of the $0$-th (or constant) term of $S(w)_{s, t}$ as $1$.
	We call $S(w)_{s, t}$ the {\it signature} of $w$ over $[s, t]$.
\end{dfn}

The following is Chen's theorem.
%\nopar
\begin{thm}[\cite{che57,lyo98}]
	\label{thm:chen}
	The process $S(w)$ satisfies
	$S(w)_{s,t}\otimes S(w)_{t,u}=S(w)_{s,u}$
	for arbitrary $0\le s\le t\le u\le T$. It also holds that
	$\log S(w)_{s,t}\in L((E))$ and therefore
	$\pi_n\left(\log S(w)_{s,t}\right)\in L^{(n)}(E)$.
	
	Moreover, the inverse of this correspondence holds, i.e.,
	for an arbitrary Lie polynomial
	$\mathcal{L}\in L^{(n)}(E)\subset T(E)$ and arbitrary $0\le s < t \le T$,
	there exists a bounded-variation path $w\in C_0^0([0,T]; \R\oplus\R^d)$
	such that $\pi_n\left(\log S(w)_{s,t}\right)=\mathcal{L}$.
\end{thm}

\begin{rem}\label{rem-piece-lin}
	Regarding the latter part,
	a stronger result is known
	\cite[Theorem 7.28]{sde-book}.
	Every Lie polynomial can be exactly
	(not approximately) represented
	as a (truncated) logarithm of
	some continuous piecewise linear path
	with a finite number of linear intervals.
\end{rem}

%\nopar

By virtue of these assertions, we see that the problem of finding paths
constructing a cubature formula is equivalent
to the problem of finding the corresponding Lie polynomials.
The following Brownian-motion version of this result is also important.

%\nopar
\begin{prop}[\cite{kus04,lyo04}]\label{bm-sig}
	Define the (Stratonovich) signature of the Brownian motion
	as an element of $T((E))=\R[[Z_0,Z_1,\ldots,Z_d]]$
	by
	\[
	S(B)_{s, t}=\sum_{n=0}^\infty
	\sum_{(i_1,\ldots,i_k)\in\A(n)\setminus\A(n-1)}
	\left(\int_{s<t_1<\cdots<t_n<t}\ccd B^{i_1}_{t_1}
	\cdots\cd B^{i_k}_{t_n}\right)
	Z_{i_1}\cdots Z_{i_k}
	\]
	for each $0\le s\le t$.
	Then, $\log S(B)_{s, t}$ is almost surely a Lie series.
\end{prop}
%%\repar

As we mainly deal with the signature over $[0, 1]$,
hereafter let $S(w)$ and $S(B)$ represent
$S(w)_{0, 1}$ and $S(B)_{0, 1}$, respectively.
We also define for each $\alpha=(i_1,\ldots,i_k)\in\mathcal{A}$,
\[
I^\alpha(w):=\int_{0<t_1<\cdots<t_k<1}\ds w^{i_1}(t_1)
\cdots\ds w^{i_k}(t_k),\quad
I^\alpha(B):=\int_{0<t_1<\cdots<t_k<1}\ccd B^{i_1}_{t_1}
\cdots\cd B^{i_k}_{t_k}.
\]
Note that we set $I^\alpha(w)=I^\alpha(B)=1$ if $\alpha=\emptyset$.

If to define $\mathcal{L}:=\pi_n\left(\log S(B)\right)$,
Indeed, we obtain the expression
\[
\E{\pi_n(S(B))}=\E{\pi_n(\exp\mathcal{L})}.
\]
As $\mathcal{L}$ is a random Lie polynomial
from the previous assertion,
roughly speaking, the generalized Tchakaloff's theorem
(Theorem \ref{MCCC:thm:gen-tchakaloff}) and the inverse statement
in Theorem \ref{thm:chen} yield the existence of a cubature formula
on Wiener space.
However, we should point out that
the surjectivity stated in Theorem \ref{thm:chen} fails
if we require that $w^0(t)$ is monotone
(so the original proof of stochastic Tchakaloff's
theorem in \cite{lyo04} should be modified).

\begin{prop}\label{counterexample}
	Let $n\ge 4$.
	Then, there exists a Lie polynomial $\mathcal{L}\in L^{(n)}(\R\oplus\R)$
	such that $\pi_n(\exp\mathcal{L})$ cannot be expressed as
	$\pi_n(S(w)_{s, t})$ for any BV path $w\in C_0^0([0, T]; \R\oplus\R)$
	with strictly monotone $w^0$.
\end{prop}

\begin{proof}
	Consider an $\R\oplus\R$-valued
	continuous BV path $w=(w^0, w^1)$ on $[0, T]$
	that starts at the origin with $w^0$ strictly increasing.
	We have
	\[
	S(w)^{(1,1,0)}_{0, T}=\int_{0<t_1<t_2<t_3<T}\dd w^1(t_1) \dd w^1(t_2) \dd w^0(t_3)
	=\int_0^T \frac{w^1(t)^2}2 \dd w^0(t).
	\]
	Because $w^0$ is strictly increasing, there exists a differentiation $\frac{\ds w^0(t)}{\ds t} \in L^1([0, T])$
	that is positive almost everywhere on $[0, T]$.
	Therefore, if $S(w)^{(1,1,0)}_{0, T}=0$ holds,
	then $w^1$ is zero almost everywhere and
	so $S(w)^{(1,0)}_{0, T}=0$ holds in particular.
	We have the same conclusion for strictly decreasing $w^0$,
	so, we have
	\[
		S(w)^{(1,1,0)}_{0, T} = 0
		\quad \Longrightarrow
		S(w)^{(1, 0)}_{0, T} = 0
	\]
	for each $w=(w^0, w^1)$ with strictly monotone $w^0$.
	
	Let $e_0, e_1 \in\R\oplus\R$ be the standard basis.
	If we consider
	\[
		Y:=\exp(e_0+[e_0, e_1]) =
		\sum_{i=0}^\infty \frac1{n!}(e_0 + e_0\otimes e_1 - e_1\otimes e_0)^{\otimes n}\in T((\R\oplus\R)),
	\]
	then its coefficient of $e_1\otimes e_1\otimes e_0$ is obviously zero, whereas that of $e_1\otimes e_0$ is $-1$.
	As $e_0+[e_0, e_1]$ is clearly a Lie polynomial,
	the proof is complete.
\qed\end{proof}

Although the surjectivity fails,
we can actually prove the existence of a cubature formula
with $w^0(t) = t$ in the following section.
We use the following well-known approximation statement
for the Brownian motion.
\begin{prop}\label{lyons-time}
	Let $n$ be a positive integer
	and $B$ be a $d$-dimensional Brownian motion.
	Then, with probability one,
	the sequence of piecewise linear paths
	$w_1, w_2, \cdots \in C_0^0([0, 1]; \R\oplus\R^d)$
	given by linearly interpolating
	$w_k(j/2^k)= B(j/2^k)$ for $j = 0, 1, \ldots, 2^k$
	satisfies
	\[
		\pi_n(S(w_k)) \to \pi_n(S(B)),\qquad
		k \to \infty.
	\]
\end{prop}

\begin{proof}
	We only give a sketch as we use arguments based on rough paths.
	If there is no time-term (the zero-th entry of the path) and $n=2$,
	then the result yields from the well-known dyadic piecewise-linear approximation for the Brownian rough path
	(see \cite[Proposition 3.6]{roughpaths} or \cite[Proposition 13.18]{sde-book}).
	Adding the time (zero-th entry) is not difficult as it is sufficiently smooth and does not affect the regularity of the rough path.
	To generalize $n$ from $n=2$, it suffices to observe the continuity of the ``Lyons lift" (also see \cite[Chapter 9]{sde-book}).
\qed\end{proof}

%%%%%%%%%%%%%%%%%%%%%%%%%%%%%%%%%%%%%%%%%%%%%%%%%%%%%%%%%%%%%%

%%%%%%%%%%%%%%%%%%%%%%%%%%%%%%%%%%%%%%%%%%%%%%%%%%%%%%%%%%%%%%
%%%%%%%%%%%%%%%%%%%%%%%%%%%%%%%%%%%%%%%%%%%%%%%%%%%%%%%%%%%%%%
%%%%%%%%%%%%%%%%%%%%%%%%%%%%%%%%%%%%%%%%%%%%%%%%%%%%%%%%%%%%%%
%%%%%%%%%%%%%%%%%%%%%%%%%%%%%%%%%%%%%%%%%%%%%%%%%%%%%%%%%%%%%%

\subsection{Proof of stochastic Tchakaloff's theorem}\label{proof-stt}
Throughout the section, we fix a positive integer $m$
and consider elements in $T^{(m)}(E)$.
Note that $T^{(m)}(E)$ can naturally be regarded in the same light
as $F:=\R^{\mathcal{A}(m)}$.
Define a set $G$ (as a subset of $F$) by
\[
G:=
\{S(w)\mid w\in C_0^0([0, 1]; \R\oplus\R^d)\ \text{is a BV path},
\ w^0(1)=1\}
\]
From Theorem \ref{thm:chen}, this coincides with the set of
$\exp(\mathcal{L})$, where $\mathcal{L}$ is a Lie polynomial such that
the coefficient of $Z_0$ is $1$.

We denote the distribution of $S(B)$ over $F$ by
$\mathrm{P}_{S(B)}$.
We shall argue the relation of $G$ and $\supp\mathrm{P}_{S(B)}$
in the following.

\begin{prop}\label{BM-full}
	It holds that $\aff \supp \mathrm{P}_{S(B)}=\aff G$.
\end{prop}

\begin{proof}
	From Proposition \ref{bm-sig}, $\supp\mathrm{P}_{S(B)}\subset \aff G$ holds
	(as $\aff G$ includes the closure of $G$).
	Therefore, it is sufficient to show $G\subset \aff\supp\mathrm{P}_{S(B)}$.
	
	As $\aff\supp\mathrm{P}_{S(B)}$
	is the intersection of all the hyperplanes, which includes
	$\supp\mathrm{P}_{S(B)}$,
	it can be represented as
	\[
	\aff\supp\mathrm{P}_{S(B)}
	=\bigcap_{({\scriptsize\bm{c}}, d)\in H}\{ \bm{v}\in F
	\mid \bm{c}^\top\bm{v}=d \},
	\]
	where $H$ is the family of all $(\bm{c}, d)\in F\times\R$
	such that $\bm{c}^\top S(B)=d$ holds almost surely.
	The problem is now reduced to the statement
	\[
	\bm{c}^\top S(B)=d\ \  \text{a.s.}\quad \Longrightarrow \quad
	\bm{c}^\top S(w)=d
	\]
	for every $w$ appearing in the definition of $G$.
	This results from the following lemma
	as $\pi_0(S(B))=\pi_0(S(w))=1$ always holds.
\qed\end{proof}

The following is the key lemma in the above proof.
We give its proof in the appendix as it is elementary.
\begin{lem}\label{long}
	Let $(c_\alpha)_{\alpha\in\mathcal{A}}
	\in\R^{\mathcal{A}}$
	be a vector whose all but finite entries are zero. 
	Then, if $\sum_{\alpha\in\mathcal{A}}c_\alpha I^\alpha(B)=0$ holds almost surely,
	\[
	\sum_{\alpha\in\mathcal{A}}c_\alpha I^\alpha(w)=0
	\]
	holds for every bounded-variation path
	$w\in C_0^0([0, 1]; \R\oplus\R^d)$ with $w^0(1)=1$.
\end{lem}

\newcommand{\I}{\mathcal{I}}

The following is a stochastic version of Tchakaloff's theorem,
which assures the existence of cubature formulas on Wiener space.
It is stated in a little stronger way than \cite[Theorem 2.4]{lyo04},
wherein the ``relative interior" did not appear.
Note that the time interval considered hereafter is $[0, 1]$.

\begin{thm}\label{st-tch}
	Let $m$ be a positive integer.
	There exist $n$ BV paths $w_1,\ldots,w_n
	\in C_0^0(\R\oplus\R^d)$
	and $n$ positive weights $\lambda_1,\ldots,\lambda_n$ whose sum is $1$
	that satisfy $n\le |\mathcal{A}(m)|$ and 
	\[
	\E{\pi_m(S(B))}
	=\sum_{i=1}^n \lambda_i \pi_m(S(w_i)).
	\]
	Moreover, if we loosen the condition to be $n\le 2|\mathcal{A}(m)|$, 
	$w_1, \ldots, w_n$ can be taken
	such that
	$\pi_m(S(w_i))$ is contained in $G$ for each $i$,
	$\aff\{ \pi_m(S(w_1)), \ldots, \pi_m(S(w_n))\}=\aff G$,
	and
	\[
	\E{\pi_m(S(B))} \in \ri \cv\{\pi_m(S(w_1)), \ldots, \pi_m(S(w_n))\}.
	\]
\end{thm}

\begin{proof}
	By virtue of Carath\'{e}odory's theorem,
	%(Theorem \ref{MCCC:thm:car}),
	the former part follows from the latter part.
	We here show the latter part.
	
	From Theorem \ref{MCCC:thm:ri}, \eqref{MCCC:lem:exp-ri}
	and Proposition \ref{BM-full},
	we can find $n$ Lie polynomials
	$\mathcal{L}_1,\ldots, \mathcal{L}_n$
	such that 
	each $\pi_m(\exp \mathcal{L}_i)$ is contained
	in $G$,
	and $\E{\pi_m(S(B))}$ is contained in the relative interior of
	their convex hull.
	Here, $n$ can actually be taken such that
	$n\le 2\dim G \ (\le 2 |\mathcal{A}(m)|)$
	because of Theorem \ref{MCCC:thm:ri}.
	From the correspondence stated in Chen's theorem
	(Theorem \ref{thm:chen}),
	we can find a desired set of paths in $C_0^0([0, 1]; \R\oplus\R^d)$.
	Note that the condition $\pi_m(\exp\mathcal{L}_i)\subset G$ implies
	that the corresponding path satisfies $w_i^0(1)=1$.
\qed\end{proof}

\begin{rem}\label{rem-lyons}
	By exploiting Proposition \ref{lyons-time},
	we can also prove the same result
	even if we require $w_i^0(t)=t$ for each $i=1,\ldots,n$ and $0\le t\le 1$.
	Although Proposition \ref{lyons-time} only asserts an approximation result,
	``relative interior" argument fills the gap.
\end{rem}

%%%%%%%%%%%%%%%%%%%%%%%%%%%%%%%%%%%%%%%%%%%%%%%%%%%%%%%%%%%%%%
%%%%%%%%%%%%%%%%%%%%%%%%%%%%%%%%%%%%%%%%%%%%%%%%%%%%%%%%%%%%%%
%%%%%%%%%%%%%%%%%%%%%%%%%%%%%%%%%%%%%%%%%%%%%%%%%%%%%%%%%%%%%%

\section{Monte Carlo approach to cubature on Wiener space}\label{chap:optCoW}

In this section, we investigate a way to construct cubature formulas,
which is based on mathematical optimization instead
of Lie algebra.
We limit the arguments to cubature formula composed of
continuous piecewise linear paths,
and propose a construction based on Monte Carlo sampling,
which is the application of \cite{hayakawa-MCCC} to our case.
We also carry out numerical experiments in concrete cases.

%\subsection{Approach from optimization}\label{sec:6-1}

Although existing constructions of cubature formulas on Wiener
space are based on Lie-algebraic equations,
we can simply regard the cubature construction
as an optimization problem.
One such way is to consider an LP problem,
which is analogous to ordinary cubature problems
treated in Section \ref{prelim}.
From this viewpoint,
we can naively generate many sample paths
and then reduce their number by using
Carath\'{e}odory-Tchakaloff subsampling.
We later see that this approach is applicable
at least theoretically (Section \ref{MCWC}).

%%%%%%%%%%%%%%%%%%%%%%%%%%%%%%%%%%%%%%%%%%%%%%%%%%%%%%%%%%%%%%
%%%%%%%%%%%%%%%%%%%%%%%%%%%%%%%%%%%%%%%%%%%%%%%%%%%%%%%%%%%%%%
%%%%%%%%%%%%%%%%%%%%%%%%%%%%%%%%%%%%%%%%%%%%%%%%%%%%%%%%%%%%%%
%%%%%%%%%%%%%%%%%%%%%%%%%%%%%%%%%%%%%%%%%%%%%%%%%%%%%%%%%%%%%%

%\subsubsection{Approximation by piecewise linear paths}

%%%%%%%%%%%%%%%%%%%%%%%%%%%%%%%%%%%%%%%%%%%%%%%%%%%%%%%%%%%%%%
%%%%%%%%%%%%%%%%%%%%%%%%%%%%%%%%%%%%%%%%%%%%%%%%%%%%%%%%%%%%%%
%%%%%%%%%%%%%%%%%%%%%%%%%%%%%%%%%%%%%%%%%%%%%%%%%%%%%%%%%%%%%%

\subsection{Signature of continuous BV paths}\label{sec:sig-BV}

In this section, we see the properties of BV paths
and their signature.
We also see that the truncated signature of continuous BV paths
can be approximated with any accuracy
by that of piecewise linear paths.

%%%%%%%%%%%%%%%%%%%%%%%%%%%%%%%%%%%%%%%%%%%%%%%%%%%%%%%%%%%%%%

%\subsubsection{Function of bounded variation}

Let $w=(w^0, w^1, \ldots, w^d)\in C_0^0([0, 1]; \R\oplus\R^d)$
be BV paths.
We define the total variation of $w$ as
\begin{equation}
	\|w\|_1
	:=\sup_{\Delta}
	\sum_{i=1}^k\max_{0\le j\le d} |w^j(t_i)-w^j(t_{i-1})|
	\left(
	=\sup_{\Delta}\sum_{i=1}^k\|w(t_i)-w(t_{i-1})\|_{\infty}
	\right),
	\label{def:TV}
\end{equation}
where $\Delta$ is the partition of $[0,1]$ by
$0=t_0<t_1<\cdots<t_k=1$ and $k$ varies in $\sup_\Delta$.
We call $w$ a BV path if $\|w\|_1<\infty$ holds.
Note that other norms are also equivalent as
the space $\R\oplus\R^d$ is finite-dimensional
though we are using the sup-norm $\|\cdot\|_\infty$ of $\R\oplus\R^d$

We can reparameterize $w$ so that it becomes Lipschitz continuous
if necessary.
Indeed, if we let $\|w|_{[s, t]}\|_1$ be
the total variation of $w$ over $[s, t]$
and
\[
\tau(t):=\frac{\|w|_{[0, t]}\|_1}{\|w\|_1}
\]
for a nonconstant $w$,
then $w\circ\tau$ is a well-defined Lipschitz path
($\tau$ becomes a nondecreasing function onto $[0, 1]$).
It is also important that the signature is invariant
under this reparametrization
(see, e.g., \cite[Proposition 1.42 and 7.10]{sde-book};
note that $w^0(t)=t$ might be
lost then even if the original path satisfies it).

Hereafter, we may assume that
there exist $d+1$ derivative functions
$f^0, f^1, \ldots, f^d \in L^\infty([0, 1]; \R)$ such that
\[
w^j(t)=\int_0^t f^j(s) \dd s,
\qquad
t\in[0, 1],\ j=0,1,\ldots,d.
\]
In this case, the total variation of $w$ can be written as
\[
\|w\|_1=\int_0^1 \max_{0\le j\le d}|f^j(s)|\dd s.
\]
%We also write $\left\|w|_{[s, t]}\right\|_1
%:=\int_s^t \max_{0\le j\le d} |f^j(r)|\dd r$.

Signature can also be represented by the derivatives as
\begin{align*}
	I^\alpha(w)
	&=\int_{0<t_1<\cdots<t_k<1}\ds w^{i_1}(t_1)\cdots\ds w^{i_k}(t_k)\\
	&=\int_{0<t_1<\cdots<t_k<1}f^{i_1}(t_1)\cdots f^{i_k}(t_k)
	\dd t_1\cdots \ds t_k.
\end{align*}
for each multiindex $\alpha=(i_1,\ldots,i_k)\in \mathcal{A}$.

As a special case of BV paths,
we are interested in (continuous) piecewise linear paths,
which are easy to implement on computers.
Let $0=s_0<s_1<\cdots<s_n=1$ be a partition of $[0, 1]$.
Then, we can define a path $w\in C_0^0([0, 1]; \R\oplus\R^d)$
which is linear on each interval $[s_{j-1}, s_j]$
($j=1, \ldots, n$)
by determining the slope vector in $\R\oplus\R^d$
at each interval.
%This corresponds to the piecewise constant
%derivative functions $f^{i_1}, \ldots, f^{i_k}$.
For the sake of computation, here we give the calculation
of the signature explicitly.
Let $g_j=(g^0_j, g^1_j, \ldots, g^d_j)$
be the slope of $w$ over $(s_{j-1}, s_j)$.
Then, for $\alpha=(i_1,\ldots,i_k)\in\mathcal{A}$,
\begin{equation}\label{eq:pwpoly}
	I^\alpha(w)
	=\sum_{\substack{
			1\le \ell_1\le \cdots \le \ell_{n-1}\le k+1\\
			\ell_0=1,\ \ell_n=k+1
	}}
	\prod_{j=1}^n
	\frac{(s_j-s_{j-1})^{\ell_j-\ell_{j-1}}}{(\ell_j-\ell_{j-1})!}
	\prod_{\ell=\ell_{j-1}}^{\ell_j-1}
	g_j^{i_\ell}
\end{equation}
holds.
We can derive this by dividing the integral domain 
into disjoint segments, which are compatible with the partition
$0=s_0<s_1<\cdots<s_n=1$.
If we adopt the notation $g_j^\alpha:=g_j^{i_1}\cdots g_j^{i_k}$
for each $\alpha\in\mathcal{A}$,
$I^\alpha(w)$ can also be written as
\[
I^\alpha(w)
=\sum_{\alpha_1*\cdots*\alpha_n=\alpha}
\prod_{j=1}^n \frac{(s_j-s_{j-1})^{|\alpha_j|}}{|\alpha_j|!}
g_j^{\alpha}.
\]
The latter expression can also be easily derived from Chen's theorem
(Theorem \ref{thm:chen}).

\subsection{Piecewise linear cubature}\label{MCWC}

%%%%%%%%%%%%%%%%%%%%%%%%%%%%%%%%%%%%%%%%%%%%%%%%%%%%%%%%%%%%%%

The following theorem assures the existence of
a cubature formula on Wiener space composed of
continuous piecewise linear paths.

\begin{thm}\label{pw-exist}
	For each positive integer $m$,
	there exist $n$ paths $w_1,\ldots, w_n\in C_0^0(\R\oplus\R^d)$,
	which are piecewise linear and
	$n$ positive weights $\lambda_1, \ldots, \lambda_n$ whose sum is $1$
	that satisfy $n\le |\mathcal{A}(m)|$ and
	\[
	\E{\pi_m(S(B))}=\sum_{i=1}^n \lambda_i\pi_m(S(w_i)).
	\]
	
	The statement still holds even if we require $w_i^0(t)=t$
	for $i=1,\ldots,n$ and $0\le t\le 1$.
\end{thm}
\begin{proof}
	From Theorem \ref{st-tch} and Remark \ref{rem-piece-lin}
	(also from Proposition \ref{lyons-time} if we want $w^0(t) = t$),
	we can easily deduce that there exist a set of
	at most $2|\mathcal{A}(m)|$ continuous piecewise linear paths
	whose truncated (by $\pi_m$) signatures convex hull contains
	$\E{\pi_m(S(B))}$ (in its relative interior).
	Rigorously, we can apply the same argument as the proof of
	Theorem \ref{MCCC:thm:main}.
	Finally, by Carath\'{e}odory's theorem,
	$n\le |\mathcal{A}(m)|$ can actually be achieved. 
\qed\end{proof}
Based on this theorem,
it is sufficient for us to look for cubature formula
within piecewise linear paths.
%As mentioned in Section \ref{sec:6-1},
%we propose two different approach to this construction problem.
%%%%%%%%%%%%%%%%%%%%%%%%%%%%%%%%%%%%%%%%%%%%%%%%%%%%%%%%%%%%%%
%\subsubsection{Construction based on Monte Carlo method}\label{MCWC}
Our approach to construction of a piecewise linear cubature
is an application of
``Monte Carlo cubature construction" \cite{hayakawa-MCCC}.
Of course, we are not able to generate
a Brownian motion and use it as a candidate for
sample points of cubature formulas,
because it is not a BV path and it cannot be implemented
on computers anyway.
However, the methods in \cite{hayakawa-MCCC} are still applicable here
as we see in the following proposition.

\begin{prop}\label{prop:pw-generate}
	Let $m$ be a positive integer.
	Then, for a sufficiently large $M$
	(the lower bound of $M$ depends on $m$),
	the following statement holds:
	\begin{quote}
		Let a sequence of continuous piecewise linear paths
		$w_1, w_2, \ldots$ be generated identically and independently.
		Assume also each $w_i$ satisfies that
		%\vspace{2mm}
		\begin{itemize}
			\item
			$w_i$ is a path that linearly connects points
			$w_i(k/M)$ $(0\le k\le M)$;
			%\vspace{2mm}
			\item one of (a) and (b) holds:
			%\vspace{1mm}
			\begin{itemize}
				\item[(a)]
				$w_i^0(t)=t$ $(0\le t\le 1)$ holds;
				%\vspace{1mm}
				\item[(b)]
				$w_i^0(0)=0$ and $w_i^0(1)=1$ hold, and
				$w_i^0(\frac{k}M)-w_i^0(\frac{k-1}M)$
				$(1\le k\le M-1)$
				are independent random variables
				and have a density on $\R$ which is positive almost everywhere;
				%\vspace{1mm}
			\end{itemize}
			\item random variables
			$w_i^j(\frac{k}M)-w_i^j(\frac{k-1}M)$
			$(1\le j \le d,\ 
			1\le k\le M)$
			are independent (also from the ones of zero-th coordinate) and have a density on $\R$, which is positive almost everywhere.
			%\vspace{2mm}
		\end{itemize}
		
		Then, with probability one,
		there exists an $N$ such that
		a subset of $\{w_1, w_2,\ldots,w_N\}$ can construct
		a cubature on Wiener space of degree $m$.
	\end{quote}
\end{prop}

\begin{proof}
	Take $M$ as large as we can find (at most) $2|\mathcal{A}(m)|$
	piecewise linear paths
	(denoted by $\tilde{w}_\ell$) with at most $M$ linear segments
	such that $\cv\{\pi_m(S(\tilde{w}_\ell))\}_\ell$ contains
	$\E{\pi_m(S(B))}$ in its relative interior.
	The existence of an $M$ is assured by the proof of
	Theorem \ref{pw-exist}
	(or see Remark \ref{rem-piece-lin},
	Theorem \ref{st-tch}, and Remark \ref{rem-lyons}).
	
	From (\ref{eq:pwpoly}), the truncated signature of each $w_i$ is
	a polynomial of random variables $w_i^j(\frac{k}M)-w_i^j(\frac{k-1}M)$
	$(0\le j \le d,\ 1\le k\le M,\ (j, k)\ne (0, M))$.
	Our assumption assures that these variables take values
	in every neighborhood of some point with a positive probability.
	In particular, it implies that
	\[
	\mathrm{P}\left(
	\|\pi_m(S(\tilde{w}_\ell))-\pi_m(S(w_i))\|<\ve
	\right)	>0,
	\]
	where $\|\cdot\|$ is the Euclidean norm on
	$T^{(m)}(E)\simeq F$,
	holds for each $i, \ell$ and $\ve>0$.
	Note that the left-hand side probability
	does not depend on $i$ by i.i.d. assumption.
	Therefore, the argument in the proof of Theorem \ref{MCCC:thm:main}
	holds here again and we obtain the desired assertion.
\qed\end{proof}

\begin{rem}
	For the scheme (b), the condition $w_i^0(1)=1$ is necessary to assure that
	each $\pi_m(S(w_i))$ is contained in $G$ defined in Section \ref{proof-stt}.
\end{rem}

Note that the generation rule of sample paths in this proposition
is just one of
infinitely many possible examples.
We may alternatively, for instance, directly generate
$w_i^j(k/M)$ independently.

Proposition \ref{prop:pw-generate} assures only the existence of
``some large" $M$ and $N$, so it might be
numerically hard to find cubature formulas from this approach.
However, it is beneficial to know that the construction can be reduced
at least to the stage of machine power.

%%%%%%%%%%%%%%%%%%%%%%%%%%%%%%%%%%%%%%%%%%%%%%%%%%%%%%%%%%%%%%
%%%%%%%%%%%%%%%%%%%%%%%%%%%%%%%%%%%%%%%%%%%%%%%%%%%%%%%%%%%%%%
%%%%%%%%%%%%%%%%%%%%%%%%%%%%%%%%%%%%%%%%%%%%%%%%%%%%%%%%%%%%%%

\subsection{Numerical experiments}\label{num-exp}

Explaining our simple numerical method
based on a Monte Carlo approach
we first describe the algorithm for computing the signature of
piecewise linear paths, and then present
our Monte Carlo approach and its result in some pairs of $(d, m)$.

\paragraph{Calculation of signature}
Note that hereafter $d$ and $m$ are regarded as already given parameters.
For positive integers $M$ and $N$,
we generate $N$ piecewise linear paths
(denoted by $w_1, \ldots, w_N$)
with $M$ intervals of time
(see Proposition \ref{prop:pw-generate}).
The time complexity of this paths generation
is $\mathrm{O}(NMd)$.
Then, we address each component of generated paths by
\[
\text{PATH}[i, j, k]:
= w_i^k\left(\frac{j}{M}\right) - w_i^k\left(\frac{j-1}{M}\right)
\qquad
(1\le i\le N,\ 1\le j\le M,\ 0\le k\le d)
.
\]
In all the experiments, we generated $\text{PATH}[i,j,k]$ (with $k\ne0$)
so that it follows the centered normal distribution of
variance $1/M$.
We set $\text{PATH}[i,j,0]=1/M$ for each $i, j$ and denote $w_i$ by just writing $\text{PATH}[i]$ for each $i$.

As we consider not so large $M$ in this study,
we calculate the signature of generated paths
by a simple dynamic programming (Algorithm \ref{alg:calc_signature}).
In the algorithm,
we calculate the signature of $w_i$ over $[0, k/M]$ for $k=1,\ldots, M$,
by using the expression (\ref{eq:pwpoly}).
The time complexity of this algorithm is $\mathrm{O}(M|\mathcal{A}(m)|^2)$,
though a pruning of possible multiindices $(\alpha, \beta)$ helps a bit.

\begin{algorithm}[h]
	\caption{Calculation of ($i$-th) signature}
	\label{alg:calc_signature}
	\begin{algorithmic}
		\Require{$M, \text{PATH}[i]$}
		\Initialize{
			$\text{SIGNATURE}[i,\emptyset]=1$ \\
			$\text{SIGNATURE}[i,\alpha]=0\ (\alpha \in \mathcal{A}_0(m))$\\
			$\text{TEMPORARY}[\alpha]\ (\alpha\in\mathcal{A}(m))$\\
			$\text{NEXT}[\alpha]\ (\alpha\in\mathcal{A}(m))$
		}
		\For{$j = 1, \ldots, M$}
		\For{$\alpha\in\mathcal{A}(m)$}
		\State{$\text{TEMPORARY}[\alpha]=\text{SIGNATURE}[i,\alpha]$}
		\State{$\text{NEXT}[\alpha]
			=\frac{1}{|\alpha|!}
			\prod_{k=1}^{|\alpha|}\text{PATH}[i, j, \alpha_k]
			$}
		\State{$\text{SIGNATURE}[i,\alpha]=0$}
		\EndFor
		\For{$(\alpha, \beta) \in\mathcal{A}(m)\times \mathcal{A}(m)$}
		\If{$\alpha * \beta\in \mathcal{A}(m)$}
		\State{$\text{SIGNATURE}[i, \alpha*\beta]
			+\!\!
			=\text{TEMPORARY}[\alpha]\cdot\text{NEXT}[\beta]$}
		\EndIf
		\EndFor
		\EndFor
		\Ensure{$\text{SIGNATURE}[i]$}
	\end{algorithmic}
\end{algorithm}

%The following algorithms for constructing cubature formulas
%are based on this calculation of signature.

\paragraph{Monte Carlo approach}

In the approach based on Monte Carlo sampling,
we simply generate many paths and determine by solving an LP problem
whether or not we can construct a cubature formula of desired degree
from generated paths.

The part of solving an LP problem was performed using IBM ILOG CPLEX Optimization
Studio (\url{https://www.ibm.com/analytics/cplex-optimizer}, version 12.10;
 CPLEX hereafter).
From Proposition \ref{prop:pw-generate}, for a sufficiently large $N$
we can construct a cubature formula using a subset of paths
$\{w_1, \ldots, w_N\}$.

We conducted experiments for six cases
$(d, m) = (2, 3), (3, 3), (4, 3), (2, 5), (3, 5), (2, 7)$
and for each $(d, m)$,
set $N=2|\mathcal{A}(m)|, 4|\mathcal{A}(m)|, 8|\mathcal{A}(m)|$,
$M=2, 4, 8, 16, 32$
and examined if we could construct a cubature formula by using CPLEX
(note that $|\mathcal{A}(m)|$ depends on $d$).
The following tables show how many times out of 10 trials
we successfully obtained cubature formula.
Blanks in the tables mean that the corresponding experiments
were not performed because we already got 10 successes out of 10 with a smaller $N$.

\begin{table}[h]
	\begin{minipage}{0.5\hsize}
		\begin{center}
			\caption{$(d, m)=(2, 3)$, $|\mathcal{A}(m)|=20$}
			\begin{tabular}{|c||c|c|c|c|c|}
				\hline 
				$N\backslash M$ & 2 & 4 & 8 & 16 & 32 \\ 
				\hhline{|=#=|=|=|=|=|}
				$2|\mathcal{A}(m)|$ & 3 & 4 & 3 & 2 & 2 \\ 
				\hline 
				$4|\mathcal{A}(m)|$ & 10 & 10 & 10 & 10 & 10 \\ 
				\hline 
				$8|\mathcal{A}(m)|$ &  &  &  &  &  \\ 
				\hline 
			\end{tabular}
		\end{center}
	\end{minipage}
	\begin{minipage}{0.5\hsize}
		\begin{center}
			\caption{$(d, m)=(3, 3)$, $|\mathcal{A}(m)|=47$}
			\begin{tabular}{|c||c|c|c|c|c|}
				\hline 
				$N\backslash M$ & 2 & 4 & 8 & 16 & 32 \\ 
				\hhline{|=#=|=|=|=|=|}
				$2|\mathcal{A}(m)|$ & 1 & 2 & 2 & 1 & 1 \\ 
				\hline 
				$4|\mathcal{A}(m)|$ & 10 & 10 & 10 & 10 & 10 \\ 
				\hline 
				$8|\mathcal{A}(m)|$ &  &  &  &  &  \\ 
				\hline 
			\end{tabular}
		\end{center}
	\end{minipage}
%\end{table}
%\begin{table}[h]
	\begin{minipage}{0.5\hsize}
		\begin{center}
			\caption{$(d, m)=(4, 3)$, $|\mathcal{A}(m)|=94$}
			\begin{tabular}{|c||c|c|c|c|c|}
				\hline 
				$N\backslash M$ & 2 & 4 & 8 & 16 & 32 \\ 
				\hhline{|=#=|=|=|=|=|}
				$2|\mathcal{A}(m)|$ & 2 & 0 & 0 & 2 & 1 \\ 
				\hline 
				$4|\mathcal{A}(m)|$ & 10 & 10 & 10 & 10 & 10 \\ 
				\hline 
				$8|\mathcal{A}(m)|$ &  &  &  &  &  \\ 
				\hline 
			\end{tabular}
		\end{center}
	\end{minipage}
	\begin{minipage}{0.5\hsize}
		\begin{center}
			\caption{$(d, m)=(2, 5)$, $|\mathcal{A}(m)|=119$}
			\begin{tabular}{|c||c|c|c|c|c|}
				\hline 
				$N\backslash M$ & 2 & 4 & 8 & 16 & 32 \\ 
				\hhline{|=#=|=|=|=|=|}
				$2|\mathcal{A}(m)|$ & 0 & 0 & 0 & 0 & 0 \\ 
				\hline 
				$4|\mathcal{A}(m)|$ & 7 & 10 & 10 & 10 & 10 \\ 
				\hline 
				$8|\mathcal{A}(m)|$ & 10 &  &  &  &  \\ 
				\hline 
			\end{tabular}
		\end{center}
	\end{minipage}
%\end{table}
%\begin{table}[h]
	\begin{minipage}{0.5\hsize}
		\begin{center}
			\caption{$(d, m)=(3, 5)$, $|\mathcal{A}(m)|=516$}
			\begin{tabular}{|c||c|c|c|c|c|}
				\hline 
				$N\backslash M$ & 2 & 4 & 8 & 16 & 32 \\ 
				\hhline{|=#=|=|=|=|=|}
				$2|\mathcal{A}(m)|$ & 0 & 0 & 0 & 0 & 0 \\ 
				\hline 
				$4|\mathcal{A}(m)|$ & 2 & 10 & 10 & 10 & 10 \\ 
				\hline 
				$8|\mathcal{A}(m)|$ & 10 &  &  &  &  \\ 
				\hline 
			\end{tabular}
		\end{center}
	\end{minipage}
	\begin{minipage}{0.5\hsize}
		\begin{center}
			\caption{$(d, m)=(2, 7)$, $|\mathcal{A}(m)|=696$}
			\begin{tabular}{|c||c|c|c|c|c|}
				\hline 
				$N\backslash M$ & 2 & 4 & 8 & 16 & 32 \\ 
				\hhline{|=#=|=|=|=|=|}
				$2|\mathcal{A}(m)|$ & 0 & 0 & 0 & 0 & 0 \\ 
				\hline 
				$4|\mathcal{A}(m)|$ & 0 & 0 & 0 & 1 & 8 \\ 
				\hline 
				$8|\mathcal{A}(m)|$ & 0 & 0 & 10 & 10 & 10 \\ 
				\hline 
			\end{tabular}
		\end{center}
	\end{minipage}
\end{table}

From these results,
one may expect that the change in $m$
is more essential than $d$ where we need larger number of partitions
and ratio $N/|\mathcal{A}(m)|$ as $m$ gets larger,
though more experiments are necessary.

%%%%%%%%%%%%%%%%%%%%%%%%%%%%%%%%%%%%%%%%%%%%%%%%%%%%%%%%%%%%%%
%%%%%%%%%%%%%%%%%%%%%%%%%%%%%%%%%%%%%%%%%%%%%%%%%%%%%%%%%%%%%%
%%%%%%%%%%%%%%%%%%%%%%%%%%%%%%%%%%%%%%%%%%%%%%%%%%%%%%%%%%%%%%

\section{Concluding remarks}\label{sum-dis}

In this paper, we have demonstrated that
 piecewise linear cubature formula can be constructed on Wiener space
through a Monte Carlo sampling and an LP problem.
Our construction is supported by the technical contribution,
which extends stochastic Tchakaloff's theorem using our characterization
of the distribution of Stratonovich iterated integrals.
We confirmed that for small pairs of $(d, m)$ our algorithm actually works
in numerical experiments.

Although we have shown that one can theoretically construct
cubature formulas of any dimension and degree,
the number of paths used in our construction only attains
the Tchakaloff bound, and therefore it requires too much computational cost
for large $(d, m)$ in practice.
Therefore, we may consider reducing the number of paths
by using additional optimization techniques.

\begin{acknowledgements}
The first author would like to thank Terry Lyons
for his insightful comments on Proposition \ref{counterexample}.
The authors would like to thank Enago (\url{www.enago.jp}) for the manuscript review and editing support.
The authors are also grateful to the reviewers
for their careful reading and constructive feedback.
This study was supported by the Japan Society for the Promotion of Science
with KAKENHI (17K14241 to K.T.).
\end{acknowledgements}


\bibliographystyle{spmpsci}      % mathematics and physical sciences
\bibliography{cite}   % name your BibTeX data base

%\begin{thebibliography}{}
%%
%% and use \bibitem to create references. Consult the Instructions
%% for authors for reference list style.
%\end{thebibliography}

\appendix
\normalsize

\section{Proof of Lemma \ref{long}}

Here, we give an elementary proof of Lemma \ref{long}.

\begin{proof}
	%Although a similar result for Ito integrals has already been
	%established in \cite{li97},
	%we here explicitly give a proof.
	This proof exploits several arguments in \cite[Chapter 5]{klo92}.
	First, we prove the result for the Ito integrals.
	Define
	\[
	\mathcal{A}^\prime:=\{\alpha\in\mathcal{A}\mid \text{$\alpha$
		contains at least one nonzero index}\}.
	\]
	We define for $t\ge0$ that $\I^\emptyset_t:=1$ and
	\[
	\I^\alpha_t:=\int_{0<t_1<\cdots<t_k<t}\ds B_{t_1}^{i_1} \cdots
	\ds B_{t_k}^{i_k}
	\]
	for each $\alpha=(i_1,\ldots,i_k)\in\mathcal{A}$.
	We also define
	$\alpha-:=(i_1,\ldots,i_{k-1})$ and 
	$s(\alpha)=i_k$ for $\alpha\ne\emptyset$.
	Note that it holds that
	\[
	\I^\alpha_t=\int_0^t \I^{\alpha-}_r \dd B_r^{s(\alpha)}.
	\]
	
	For a multiindex $\alpha=(i_1,\ldots,i_k)\in\mathcal{A}$,
	let $\alpha^+$ be the sequence defined by nonzero indices of $\alpha$.
	For example, if $\alpha=(0, 2, 0, 1, 1, 0)$, then $\alpha^+=(2, 1, 1)$.
	
	Let $u=(u_t)_{t\ge0}$ be a progressively measurable
	and second mean integrable stochastic process.
	Then, for each index $i\ne0$,
	\[
	\int_0^t\left(\int_0^su_r \dd B^i_r\right)\ds s
	=\int_0^t\left(\int_0^t u_r 1_{\{s>r\}}(r, s)\dd s\right) \ds B^i_r
	=\int_0^t (t-r)u_r \dd B_r^i
	\]
	holds.
	By using Ito isometry and this relation repeatedly,
	we can show that
	\begin{equation}
		\E{\I^\alpha_t\I^\beta_t}=0,\qquad \text{if
			$\alpha, \beta\in \mathcal{A}^\prime$ and $s(\alpha^+)\ne s(\beta^+)$}.
		\label{indep-1}
	\end{equation}
	Indeed, in the case $k$-times zeros appear in the suffix of $\alpha$,
	we can show inductively
	\begin{align}
		&\int_{0<t_1<\cdots<t_k<t}\left(\int_0^{t_1} u_r \dd B_r^i\right) \ds t_1
		\cdots\ds t_k \nonumber\\
		&=\cdots=\int_{0<t_j<\cdots<t_k<t}\left(\int_0^{t_j}
		\frac{(t_j-r)^{j-1}}{(j-1)!} u_r \dd B^i_r
		\right)\ds t_j \cdots \ds t_k\nonumber\\
		&=\cdots=\int_0^t\frac{(t-r)^{k}}{k!}u_r\dd B_r^i,
		\label{indep-repeat}
	\end{align}
	so, we can finally apply the following Ito isometry:
	\begin{quote}
		for progressively measurable
		and second mean integrable stochastic process $u$ and $v$, it holds that
		\[
		\E{\left(\int_0^tu_s\dd B_s^i\right)\left(\int_0^tv_s\dd B_s^j\right)}
		=\delta^{ij}\E{\int_0^t u_sv_s \dd s},
		\]
		where $i,j\in\{1,\ldots,d\}$ and $\delta^{ij}$ is Kronecker's delta.
	\end{quote}
	
	We then prove a stronger assertion than (\ref{indep-1}), which states
	\begin{equation}
		\E{\I^\alpha_t\I^\beta_t}=0,\qquad
		\text{if $\alpha\text{ or }\beta\in\mathcal{A}^\prime$, and
			$\alpha^+\ne \beta^+$}.
		\label{indep-2}
	\end{equation}
	Let $\alpha=\tilde{\alpha}*(i, \underbrace{0, \ldots,0}_{k \text{ times}})$
	and $\beta=\tilde{\beta}*(j, \underbrace{0, \ldots,0}_{\ell \text{ times}})$.
	It suffices to consider the case $i=j\ne0$, and we have
	\begin{align}
		\E{\I^\alpha_t\I^\beta_t}
		&=\E{\left(\int_0^t
			\frac{(t-s)^k}{k!}
			\I^{\tilde{\alpha}}_s \dd B_s^i\right)
			\left(\int_0^t
			\frac{(t-s)^\ell}{\ell!}
			\I^{\tilde{\beta}}_s \dd B_s^i\right)}\nonumber\\
		&=\E{\int_0^t\frac{(t-s)^{k+\ell}}{k!\ell!}
			\I^{\tilde{\alpha}}_s\I^{\tilde{\beta}}_s
			\dd s
		}\nonumber\\
		&=\int_0^t \frac{(t-s)^{k+\ell}}{k!\ell!}
		\E{\I^{\tilde{\alpha}}_s\I^{\tilde{\beta}}_s} \dd s.\label{indep-isometry}
	\end{align}
	Therefore, by an inductive argument, it only remains to prove that
	$\E{\I^\alpha_t\I^\beta_t}=0$ in the case $\alpha$ or $\beta$ only contain zeros
	and $\alpha^+\ne\beta^+$,
	but this case is trivial as the one is a constant and
	the other's expectation becomes zero
	from (\ref{indep-repeat}).
	We now have completed the proof of (\ref{indep-2}).
	
	From these results, we have the decomposition
	\[
	\E{\left(
		\sum_{\alpha\in\mathcal{A}} c_\alpha\I^\alpha_t
		\right)^2}
	=\sum_{\beta\in\mathcal{A}^+} \E{
		\left(\sum_{\alpha\in\mathcal{A},\ \alpha^+=\beta}
		c_\alpha\I^\alpha_t
		\right)^2
	},
	\]
	where all but finite $c_\alpha$ equal to zero and $\mathcal{A}^+$ denotes
	the set of all multiindices that contain no zeros.
	Therefore, it suffices to consider sums of the form
	$\sum_{i=1}^n c_i \mathcal{I}^{\alpha_i}_t$,
	with $c_1,\ldots,c_n\in\R$ and $\alpha_1^+=\cdots=\alpha_n^+\ne\emptyset$.
	Note that $\alpha_1,\ldots,\alpha_n$ can be taken pairwaise different.
	One may assume this sum almost surely equals to zero.
	Next, we prove that
	$\sum_{i=1}^nc_iI^{\alpha_i}(w)_{0, t}$ is also zero for
	the $w$ with $w(s)=s$
	($I^\alpha(w)_{s, t}$ is
	$S(w)_{s, t}$ entry corresponding to the multiindex $\alpha$).
	Let each $\alpha_i$ have the form
	$\alpha_i=\beta_i * (j, \underbrace{0,\ldots,0}_{k_i\text{ times}})$,
	where $j\ne0$ is an index independent of $i$.
	
	From (\ref{indep-repeat}),
	we have
	\begin{align}
		\E{\left(\sum_{i=1}^n c_i\I^{\alpha_i}_t\right)^2}
		&=\E{\left(\int_0^t
			\sum_{i=1}^n\frac{(t-r)^{k_i}}{k_i!}c_i
			\I^{\beta_i}_r \dd B_r^j
			\right)^2}\nonumber\\
		&=\E{
			\int_0^t \left(
			\sum_{i=1}^n\frac{(t-r)^{k_i}}{k_i!}c_i
			\I^{\beta_i}_r
			\right)^2
			\ds r
		}\nonumber\\
		&=\int_0^t \E{
			\left(
			\sum_{i=1}^n\frac{(t-r)^{k_i}}{k_i!}c_i
			\I^{\beta_i}_r
			\right)^2
		} \ds r.\label{eq:rev-1}
	\end{align}
	If the left-hand side is zero,
	the integrand in the right-hand side is also zero for all $0\le r\le t$
	because it is continuous in $r$.
	
	Here, we note that the deterministic counterpart
	of (\ref{indep-repeat}) holds.
	Indeed,
	if $u=(u_t)_{t\ge0}$ is a continuous $\R$-valued path,
	then for a $w\in C_0^0([0, 1]; \R\oplus\R^d)$ with $w^0(t)=t$,
	we have
	\begin{align}
		&\int_{0<t_1<\cdots<t_k<t}\left(\int_0^{t_1} u_r \dd w_r^i\right) \ds t_1
		\cdots\ds t_k \nonumber\\
		&=\cdots=\int_{0<t_j<\cdots<t_k<t}\left(\int_0^{t_j}
		\frac{(t_j-r)^{j-1}}{(j-1)!} u_r \dd w^i_r
		\right)\ds t_j \cdots \ds t_k\nonumber\\
		&=\cdots=\int_0^t\frac{(t-r)^{k}}{k!}u_r\dd w_r^i.
		\label{indep-repeat-2}
	\end{align}
	\begin{align*}
	\end{align*}
	Therefore, we obtain
	\begin{equation}
		\sum_{i=1}^nc_i I^{\alpha_i}(w)_{0, t}
		=\int_0^t \left(
		\frac{(t-r)^{k_i}}{k_i!}
		c_i I^{\beta_i}(w)_{0, r}
		\right)\ds w_r^j.
		\label{rev-3}
	\end{equation}
	As we have already obtained
	$\sum_{i=1}^n\frac{(t-r)^{k_i}}{k_i!}c_i
	\I^{\beta_i}_r=0$ for $0\le r\le t$,
	the problem reduces to that of $\beta_1, \ldots, \beta_n$.
	Therefore, by inductive arguments with respect to $|\alpha_1^+|$
	we can prove that
	\[
	\sum_{i=1}^nc_i\mathcal{I}_t^{\alpha_i}=0
	\quad \Longrightarrow \quad 
	\sum_{i=1}^nc_iI^{\alpha_i}(w)_{0, t}=0,
	\]
	where the base step is trivial as it only contains the integral
	with respect to time.
	
	Though we have to extend this result when $t=1$ to all the $w$ satisfying
	only $w^0(1)=1$,
	it can be proved by using perturbation arguments.
	More precisely, we can prove the result for piecewise linear paths
	with $w^0(1)=1$ as $\sum_ic_iI^{\alpha_i}(w)$ is a polynomial 
	of increments in the path that equals to zero at infinitely many points
	(if $w^0$ is monotone increasing,
	then we can reparametrize it so that $w^0(t)=t$).
	Indeed, proving for piecewise linear paths is sufficient
	\cite[Theorem 7.28]{sde-book}.
	
	It remains to modify the result for Stratonovich integrals.
	The relation between multiple Stratonovich integrals and Ito integrals
	is known as follows \cite{klo92}:
	\begin{align*}
		I^\emptyset(B)_{0,t}&=\I^\emptyset_t=1,\\
		I^{(i)}(B)_{0, t}&=\I^{(i)}_t \quad (i=0, 1,\ldots,d),\\
		I^\alpha(B)_{0, t}
		&=\int_0^t I^{\alpha-}(B)_{0, r} \dd B_r^{s(\alpha)}
		+\frac12\cdot1_{\{s(\alpha)=s(\alpha-)\ne0\}}
		\int_0^t I^{(\alpha-)-}(B)_{0, r} \dd r
		\quad (|\alpha|\ge2).
	\end{align*}
	Therefore, each $I^\alpha(B)_{0, t}$ is represented as a positive combination
	of $\I^\beta_t$ such that
	$\beta$ can be acquired by replacing two consecutive same nonzero
	indices by a zero some times.
	%	If we write $\beta \preceq \alpha$ when we can make $\beta$ from $\alpha$
	%	following the above operation,
	%	$(\mathcal{A}, \preceq)$ becomes a partially-ordered set.
	%	Consider a linear combination
	%	$\sum_{i=1}^n c_iI^{\alpha_i}(B)_{0, t}$
	%	with $c_i\ne0$ and pairwise-different $\alpha_i\in\mathcal{A}^\prime$.
	%	When we rewrite the sum by Ito integrals,
	%	$\I^{\alpha_i}_t$ with the multiindex $\alpha_i$, which is maximal among
	%	$\alpha_1,\dots,\alpha_n$ with respect to $\preceq$
	%	appears exactly once.
	%	From this fact, the linearly-independence of
	%	multiple Stratonovich integrals also results from that of Ito integrals.
	The assumption $\sum_{i=1}^n c_iI^{\alpha_i}(B)=0$ holds almost surely
	($c_i\ne0$ for all $i$).
	We shall again prove $\sum_{i=1}^nc_iI^{\alpha_i}(w)$
	by inductive arguments.
	
	Let $i_0\in\arg\max_i|\alpha_i^+|$ and $\beta:=\alpha_{i_0}^+$.
	Then, by the above expansion, $\sum_{i=1}^nc_iI^{\alpha_i}(B)$
	can be rewritten as a sum of Ito integrals.
	In particular,
	\[
	\sum_{i=1}^nc_iI^{\alpha_i}(B)
	-
	\sum_{i:\,\alpha_i^+=\beta}c_i\mathcal{I}^{\alpha_i}_1
	\]
	is represented as a weighted sum of $\mathcal{I}^\alpha_1$
	with $\alpha^+\ne\beta$.
	Therefore, $\sum_{i:\,\alpha_i^+=\beta}c_i\mathcal{I}^{\alpha_i}_1=0$
	holds almost surely, and so $\sum_{i:\,\alpha_i^+=\beta}
	c_iI^{\alpha_i}(w)=0$ holds for all the valid $w$.
	
	From Theorem \ref{thm:chen} and Proposition \ref{bm-sig},
	with probability one there exists some (random)
	$w$ such that $\pi_n(S(B))=\pi_n(S(w))$,
	where sufficiently large $n$ is taken.
	By using such $w$, we obtain
	\[
	\sum_{i:\, \alpha_i^+=\beta}c_iI^{\alpha_i}(B)
	=\sum_{i:\, \alpha_i^+=\beta}c_iI^{\alpha_i}(w)
	=0.
	\]
	Therefore, if $\sum_{i=1}^nc_iI^{\alpha_i}(B)=0$ holds almost surely,
	then we can prove inductively (with respect to some order over
	$\mathcal{A}^+$),
	for each $\beta\in\mathcal{A}^+$ and $w$,
	\[
	\sum_{i:\,\alpha_i^+=\beta}c_iI^{\alpha_i}(w)=0,\qquad
	\sum_{i:\,\alpha_i^+=\beta}c_iI^{\alpha_i}(B)=0.
	\]
	By considering the sum, we obtain
	$\sum_{i=1}^nc_iI^{\alpha_i}(w)=0$ and the proof is completed.
	\qed\end{proof}

\begin{rem}
	In the last part of the above proof,
	we have essentially used the assertion
	(actually the inverse also follows from the above proof)
	\[
	\sum_\alpha c_\alpha\mathcal{I}^\alpha_t = 0
	\quad \Longrightarrow \quad
	\sum_\alpha c_\alpha I(B)_{0, t}=0.
	\]
	Although we have proved it via Lie-algebraic arguments, which exploit
	Theorem \ref{thm:chen} and Proposition \ref{bm-sig},
	it can be directly proved by repeatedly
	using the relations (\ref{indep-repeat})
	and (\ref{eq:rev-1}).
\end{rem}

\end{document}